\documentclass[a4paper,11pt]{amsart}
\usepackage{amssymb}
\usepackage{amsfonts}
\usepackage{amsmath}
\usepackage{amsthm}
\usepackage{cite}
\usepackage[dvipsnames]{xcolor}

\usepackage{graphicx,subfigure}
\newtheorem{theorem}{Theorem}[section]
\newtheorem{proposition}[theorem]{Proposition}
\newtheorem{corollary}[theorem]{Corollary}

\newtheorem{lemma}[theorem]{Lemma}

\theoremstyle{definition}

\theoremstyle{remark}
\newtheorem{remark}[theorem]{Remark}

\numberwithin{equation}{section}

    {\color{BlueViolet}}

\newcommand{\eps}{\epsilon}

\def\C{\mathbb{C}}

\def\N{\mathbb{N}}
\def\RR{\mathbb{R}}
\def\R{\mathbb{R}}

\newcommand{\zN}{{\bf z({\text{\tiny $N$}})}}
\newcommand{\zE}{{\bf z({\text{\footnotesize $3$}})}}

\newcommand{\cA}{{\mathcal A}}
\newcommand{\cB}{{\mathcal B}}
\newcommand{\cC}{{\mathcal C}}

\newcommand{\cF}{{\mathcal F}}

\newcommand{\cL}{{\mathcal L}}

\newcommand{\cR}{{\mathcal R}}
\newcommand{\cS}{{\mathcal S}}

\def\wh{\widehat}
\newcommand\minus\backslash
\newcommand{\id}{{\rm id}}

\newcommand\lan\langle
\newcommand\ran\rangle

\newcommand{\supp}{\operatorname{supp}}

\renewcommand\leq\leqslant
\renewcommand\geq\geqslant
\newlength{\intwidth}

\newcommand{\norm}[1]{\|#1\|}

\renewcommand{\phi}{\varphi}

\begin{document}

\title[Nonlinear Helmholtz equation ]{Complex solutions and stationary scattering for the nonlinear Helmholtz equation
   }

 \author{Huyuan Chen}
 \address{Department of Mathematics, Jiangxi Normal University, Nanchang,\\
Jiangxi 330022, PR China}
 \email{ chenhuyuan@yeah.net}

 \author{Gilles Ev\'equoz }
 \address{School of Engineering, University of Applied Sciences of Western Switzerland, Route du Rawil 47,\\
1950 Sion, Switzerland}
 \email{ gilles.evequoz@hevs.ch}

 \author{Tobias Weth}
 \address{Goethe-Universit\"{a}t Frankfurt, Institut f\"{u}r Mathematik, Robert-Mayer-Str. 10\\
D-60629 Frankfurt, Germany }
 \email{ weth@math.uni-frankurt.de}

%
%
\begin{abstract}
We study a stationary scattering problem related to the nonlinear Helmholtz equation
$
- \Delta u -k^2 u = f(x,u) \ \ \text{in $\R^N$,}
$
where $N \ge 3$ and $k>0$. For a given incident free wave $\phi \in L^\infty(\R^N)$, we prove the existence of complex-valued solutions of the form $u=\varphi+u_{\text{sc}}$, where $u_{\text{sc}}$ satisfies the Sommerfeld outgoing radiation condition. Since neither a variational framework nor maximum principles are available for this problem, we use topological fixed point theory and global bifurcation theory to solve an associated integral equation involving the Helmholtz resolvent operator. The key step of this approach is the proof of suitable a priori bounds.  
\end{abstract}
\maketitle

\section{Introduction}

A basic model for wave propagation in an ambient medium with nonlinear response is provided by the nonlinear wave equation
\begin{equation}\label{nlw0}
\frac{\partial^2 \psi}{\partial t^2}(t,x) - \Delta \psi(t,x)  =f(x,\psi(t,x)), \qquad (t,x)\in\R\times\R^N.
\end{equation}
Considering nonlinearities of the form $f(x,\psi)=g(x,|\psi|^2)\psi$, 
where $g$ is a real-valued function, the time-periodic ansatz 
\begin{equation}
\label{psi-ansatz}
	\psi(t,x)=e^{-i k t}u(x), \qquad k >0
\end{equation}
leads to the nonlinear Helmholtz equation
\begin{equation}\label{nlh-1}
- \Delta u -k^2 u = f(x,u) \qquad \text{in $\R^N$.}
\end{equation} 
Assuming in this model that nonlinear interactions occur only locally in space. 
 we are lead to restrict our attention to nonlinearities $f \in C(\R^N \times \C,\C)$ with $\lim \limits_{|x| \to \infty}f(x,u)=0$ for every $u \in \R$. 
The stationary scattering problem then consists in analyzing solutions of the form $u=\varphi+u_{\text{sc}}$, where $\varphi$ is a solution of the homogeneous Helmholtz equation $-\Delta \phi - k^2 \phi=0$ and $u_{\text{sc}}$ obeys the Sommerfeld outgoing radiation condition
\begin{equation}\label{sommerfeld-1}
	r^{\frac{N-1}{2}}\left|\frac{\partial u_{\text{sc}}}{\partial r}-iku_{\text{sc}} \right|\to 0\quad\text{as }r=|x| \to\infty
\end{equation}
or a suitable variant of it. The function $\varphi$ represents a given {\em incident free wave} whose interaction with the nonlinear ambient medium gives 
rise to a scattered wave $u_{\text{sc}}$. Usually, $\varphi$ is chosen as a plane wave 
\begin{equation}
  \label{eq:plane-wave}
\varphi(x)=e^{ik\:x\cdot \xi},\qquad  \xi\in S^{N-1}
\end{equation}
or as superposition of plane waves. To justify the notions of incident and scattered wave, let us assume for the moment that the nonlinearity is compactly supported in the space variable $x$. Then $u_{\text{sc}}$ has the asymptotics $u_{\text{sc}}(x)= r^{\frac{1-N}{2}}e^{i kr}g(\frac{x}{|x|})+ o(r^{\frac{1-N}{2}})$ as $r= |x| \to \infty$ with a function $g: S^{N-1} \to \C$ (see \cite[Theorem 2.5]{colton-kress} and \cite[Proposition 2.6]{EW1}). For incident plane waves $\phi$ as in \eqref{eq:plane-wave}, this leads to the asymptotic expansion 
\begin{equation*}
	\psi(t,x)=e^{i k (x \cdot \xi-t)}+ r^{\frac{1-N}{2}}e^{i k(r-t)}g(\frac{x}{|x|})
+o(r^{\frac{1-N}{2}}) \qquad \text{as $r= |x| \to \infty$}
\end{equation*}
uniformly in $t \in \R$ for the corresponding time periodic solution given by the ansatz~\eqref{psi-ansatz}. This expansion clearly shows the asymptotic decomposition of the wave function $\psi$ in two parts, of which one is propagating with constant speed $k$ in the given direction $\xi$ and the other part is outward radiating in the radial direction. For a more detailed discussion of the connection of notions of stationary and dynamical scattering, we refer the reader to \cite{komech} and the references therein.  

\qquad In the (affine) linear case $f(x,u)=a(x) u+b(x)$, both the forward and the inverse stationary scattering problem
have been extensively studied and are reasonably well understood from a functional analytic point of view (see e.g. \cite{colton-kress} and the references therein).  In contrast, the nonlinear setting remains widely unexplored, although it appears in important models driven by applications and therefore is receiving fastly growing attention in recent years. Specifically, we mention the modeling of propagation and scattering of electromagnetic waves in localized nonlinear Kerr media as considered e.g. in \cite{fibich-tsynkov:2005,baruch-fibich-tsynkov:2009,wu-zou}. In this context, the nonlinear Helmholtz equation arises from a reduction of Maxwell's equations in the case of a linearly polarized electric field after elimination of the corresponding magnetic field. As noted in \cite{wu-zou}, this leads to a special case of equation (\ref{nlh-1}) given by
  $$
- \Delta u -k^2 u = \rho 1_{\Omega}|u|^2 u \qquad \text{in $\R^N$.}  
  $$
Here $\Omega \subset \R^N$ is the support of the nonlinear Kerr medium and $\rho$ is the Kerr constant given as quotient of the Kerr coefficient of the medium and the index of refraction of the ambient homogeneous medium. Both from a theoretical and an applied point of view, it is of great interest to understand self-focusing and scattering effects of laser beams interacting with localized nonlinear media, and computational approaches to these questions have been developed e.g. in \cite{fibich-tsynkov:2005,baruch-fibich-tsynkov:2009,wu-zou}.    

From a theoretical point of view, the current understanding of the stationary scattering problem for (\ref{nlh-1}) is 
mainly restricted to the case of small incident waves $\phi$ which
can be reduced to a perturbation of an associated linear problem in suitable function spaces. In this case, existence and well-posedness results have been obtained by Guti\'errez \cite{G}, Jalade \cite{J} and Gell-Redman et al. \cite{Gell-Red}. In \cite{J}, the scattering problem is studied for a small incident
plane wave and a family of compactly supported nonlinearities in dimension $N=3$. 
The main result in \cite{G} yields, in dimensions $N=3,4$, the existence of solutions
 to the scattering problem with small incident Herglotz wave $\phi$ and cubic power nonlinearity. We recall that a Herglotz wave is a function of the type 
\begin{equation}
   \label{eq:18}
x \mapsto \phi(x):= \int_{S^{N-1}} e^{i k(x\cdot \xi)} g (\xi)\,d \sigma(\xi) \qquad
\text{for some function $g \in L^2(S^{N-1})$.}
\end{equation}
Since plane waves of the form (\ref{eq:plane-wave}) cannot be written in this way, they are not admitted in \cite{G}. On the other hand, no asymptotic decay of the nonlinearity is required for the approach developed in \cite{G}. This is also the case for the approach in \cite{Gell-Red}, where more general nonlinearities are considered, while the class of admissible incident Herglotz waves $\phi$ is restricted by assuming smallness measured in higher Sobolev norms on $S^{N-1}$.

\qquad The main reason for the smallness assumption in the papers \cite{G,J,Gell-Red} is the use of contraction mappings together with resolvent estimates for the Helmholtz operator.  
The main aim of this paper is to remove this smallness assumption by means of different tools from nonlinear analysis and new a priori estimates on the set of solutions. More precisely, for a given solution $\phi \in L^\infty(\R^N)$ of the homogeneous Helmholtz equation 
$\Delta \phi + \phi= 0$ which we shall refer to as {\em incident free wave} in the following, we wish to find solutions of (\ref{nlh-1}) of the form 
 $u = \phi + u_{sc} \in L^\infty(\R^N)$ with $u_{sc}$ satisfying  (\ref{sommerfeld-1}) or a suitable variant of this radiation condition. This problem can be reduced to an integral equation involving the  
Helmholtz resolvent operator $\cR_k$, which is formally given as a convolution $\cR_k f= \Phi_k * f$ with the fundamental solution 
\begin{equation}\label{eqn:fund_sol}
\Phi_k : \R^N \setminus \{0\} \to \C, \qquad   \Phi_k(x)=\frac{i}{4} \Bigl(\frac{k}{2\pi |x|}\Bigr)^{\frac{N-2}{2}}H^{(1)}_{\frac{N-2}{2}}(k|x|)
\end{equation}
associated to (\ref{sommerfeld-1}). Here $H^{(1)}_{\frac{N-2}{2}}$ is the 
Hankel function of the first kind of order $\frac{N-2}{2}$, see e.g. \cite{AS}. It is easy to see from the asymptotics of $H^{(1)}_{\frac{N-2}{2}}$ that $\Phi_k$ satisfies (\ref{sommerfeld-1}), and the same is true for $u:= \cR_k h= \Phi_k * h$ e.g. in the case where $h \in L^\infty(\R^N)$ has compact support. 

\qquad By the estimate in \cite[Theorem 8]{G} and the remark following it, an integral variant of (\ref{sommerfeld-1}) is available under weaker assumptions on $h$. More precisely, if $N=3,4$ and $1 < p  \le \frac{2(N+1)}{N+3}$ or $N \ge 5$ and $\frac{2N}{N+4} \le p \le \frac{2(N+1)}{N+3}$, then, for $h \in L^p(\R^N)$, the function $u= \cR_k h$ is a well-defined strong solution of the inhomogeneous Helmholtz equation $-\Delta u - k^2 u = h$ satisfying the following variant of the Sommerfeld outgoing radiation condition: 
\begin{equation}\label{eqn:sommerfeld1-averaged}
\lim_{R\to\infty}\frac{1}{R} \int_{B_R}\left|\nabla
u(x)-iku(x)\frac{x}{|x|} \right|^2\, dx=0.
\end{equation}
 Hence, under appropriate assumptions on the nonlinearity $f$, we are led to study the integral equation 
\begin{equation}\label{nlh-1-integral}
u = \cR_k(N_f(u))+ \phi \qquad \text{in $L^\infty(\R^N)$}
\end{equation} 
for a given incident free wave $\phi \in L^\infty(\R^N)$. Here $N_f$ is the substitution operator associated to $f$ given by $N_f(u)(x):= f(x,u(x))$. 

\qquad To state our main results we need to introduce some more notation. It is convenient to define $\langle x \rangle  = (1+|x|^2)^{\frac{1}{2}}$ for $x \in \R^N$.
For $\alpha\in\R$ and a measurable subset $A \subset \R^N$, we consider the Banach space $L^\infty_\alpha(A)$ of measurable functions $w: A \to \C$ with
 $$
\norm{w}_{L^\infty_\alpha(A)}:=\| \langle \,\cdot\, \rangle^{\alpha} w\|_{L^\infty(A)
} <+\infty.
$$
In particular, $L^\infty(A)=L^\infty_0(A)$. In the case $A= \R^N$, we merely write $\|\cdot\|_{L^\infty_\alpha}$ in place of $\|\cdot\|_{L^\infty_\alpha(\R^N)}$. For subspaces of real-valued functions, we use the notations $L^p(A,\R)$ for $1 \le p \le \infty$ and $L^\infty_\alpha(A,\R)$. We first note the following preliminary observation regarding properties of the resolvent operator $\cR_k$.

\begin{proposition}\label{resolvent-compact-and-continuous}
Let $N\geq2$,  $\alpha>\frac{N+1}{2}$ and  $\tau(\alpha)$ be defined by 
\begin{align} \label{exp 1}
\tau(\alpha) &=
\begin{cases} \alpha-\frac{N+1}{2}\quad&{\rm if}\ \, \frac{N+1}{2}<\alpha<N, \\[1.5mm]
 \frac{N-1}{2}\quad&{\rm if}\ \, \alpha\geq N.
 \end{cases}
\end{align}
Then we have 
\begin{equation}
  \label{eq:kappa-sigma-finite}
\kappa_{\alpha}:= \sup \Big\{ \bigl \| |\Phi_k| * w \bigr \|_{L^\infty_{\tau(\alpha)}}:\: w \in L^\infty_\alpha(\RR^N),\: \| w\|_{L^\infty_{\alpha }}=1 \Big\}
< \infty,
\end{equation}
so $\cR_k$ defines a bounded linear map $L^\infty_\alpha(\RR^N) \to L^\infty_{\tau(\alpha)}(\RR^N)$. Moreover:
\begin{enumerate}
\item[(i)] The resolvent operator defines a compact linear map $\cR_k: L^\infty_\alpha(\RR^N) \to L^\infty(\RR^N)$.\\
\item[(ii)] If $\alpha > \frac{N(N+3)}{2(N+1)}$ and $h \in L^\infty_{\alpha}(\R^N)$, then the function $u:= \cR_k h$ is a strong solution of $-\Delta u - k^2 u = h$ satisfying~(\ref{eqn:sommerfeld1-averaged}). If $\alpha > N$, then $u$ satisfies~(\ref{sommerfeld-1}).
\end{enumerate}
\end{proposition}

\qquad Our first main existence result is concerned with linearly bounded nonlinearities $f$. 
 
\begin{theorem} \label{W teo 1-sublinear}
Let, for some $\alpha>\frac{N+1}{2}$, the nonlinearity $f: \R^N \times \C \to \C$ be a continuous function satisfying 
\begin{equation}
  \label{eq:assumption-f1}
 \sup_{|u|\le M,x \in \R^N} \langle x \rangle^{\alpha}|f(x,u)|< \infty \qquad \text{for all $M >0$.}   
\end{equation}
Moreover, suppose that \underline{one} of the following assumptions is satisfied: 
\begin{enumerate}
 \item[$(f_1)$] The nonlinearity is of the form $f(x,u)= a(x) u + b(x,u)$ with $a \in L^\infty_\alpha(\R^N,\R)$ and 
$$
\sup \limits_{|u|\le M, x \in \R^N} \langle x \rangle^{\alpha} |b(x,u)|= o(M) \qquad \text{as $M \to +\infty.$} 
$$
\item[$(f_2)$] There exists $Q,b \in L^\infty_{\alpha}(\R^N,\R)$ with $\|Q\|_{L^\infty_\alpha} < \frac{1}{\kappa_\alpha}$, where $\kappa_\alpha$ is given in (\ref{eq:kappa-sigma-finite}), and  
\begin{equation*}
|f(x,u)|\leq Q(x) |u|+b(x)\qquad\text{for all }\, (x,u)\in \R^N \times \C.
\end{equation*}
\end{enumerate}
Then, for any given solution $\phi \in L^\infty(\R^N)$ of the homogeneous Helmholtz equation $\Delta \phi + k^2 \phi = 0$, the equation (\ref{nlh-1-integral}) admits a solution $u \in L^\infty(\R^N)$. 
\end{theorem}

\begin{remark}
(i) In many semilinear elliptic problems with asymptotically linear nonlinearities as in assumption $(f_1)$, additional nonresonance conditions have to be assumed to guarantee a priori bounds which eventually lead to the existence of solutions. This is not the case in the present scattering problem. We shall establish a priori bounds merely as a consequence of $(f_1)$ by means of suitable nonexistence results for solutions of the linear Helmholtz equation satisfying the radiation condition~(\ref{eqn:sommerfeld1-averaged}). The key assumption here is that the function $a$ in $(f_1)$ is real-valued. 

\qquad $(ii)$ Theorem~\ref{W teo 1-sublinear} leaves open the question of uniqueness of solutions to (\ref{nlh-1-integral}). In fact, under the sole assumptions of Theorem~\ref{W teo 1-sublinear}, uniqueness is not to be expected. If, however, for some $\alpha>\frac{N+1}{2}$, the nonlinearity $f \in C(\R^N \times \R, \R)$ satisfies (\ref{eq:assumption-f1}) and the Lipschitz condition 
\begin{equation}
  \label{eq:assumption-f1-lipschitz}
\ell_\alpha:= \sup \Bigl \{ \langle x \rangle^{\alpha} \, \Bigl|\frac{f(x,u)-f(x,v)}{u-v}\Bigr|\::\: u,v \in \R, \: x \in \R^N \Bigr\} < \frac{1}{\kappa_\alpha},   
\end{equation}
then the contraction mapping principle readily yields the existence of a unique solution $u \in L^\infty(\R^N)$ of (\ref{nlh-1-integral}) for given $\phi \in L^\infty(\R^N)$, see Theorem~\ref{theo-uniqueness} below.
\end{remark}

\qquad Next we turn our attention to superlinear nonlinearities which do not satisfy $(f_1)$ or $(f_2)$. Assuming additional regularity estimates for $f$, we can still prove the existence of solutions of (\ref{nlh-1-integral}) in the case where $\|\phi\|_{L^\infty(\R^N)}$ is small. More precisely, we have the following. 

\begin{theorem} \label{teo-implicit-function}
Let, for some $\alpha>\frac{N+1}{2}$, the nonlinearity $f: \R^N \times \C \to \C$ be a continuous function satisfying 
(\ref{eq:assumption-f1}). Suppose moreover that the function $f(x,\cdot):\C \to \C$ is real differentiable for every $x \in \R^N$, and that $f':= \partial_u f: \R^N \times \C \to \cL_{\R}(\C,\C)$ is a continuous function satisfying 
\begin{equation}
  \label{eq:assumption-f1-diff}
 \sup_{|u|\le M,x \in \R^N} \langle x \rangle^{\alpha}\|f'(x,u)\|_{\cL_{\R}(\C,\C)}< \infty.   
\end{equation}
Finally, suppose that $f(x,0)=0$ and $f'(x,0)= 0 \in \cL_\R(\C,\C)$ for all $x \in \R^N$. 

\qquad Then there exists open neighborhoods $U,V \subset L^\infty(\R^N)$ of zero with the property that for every $\phi \in V$ there exists a unique solution $u= u_\phi \in U$ of (\ref{nlh-1-integral}). Moreover, the map $V \to U$, $u \mapsto u_\phi$ is of class $C^1$. 
\end{theorem}

\qquad The proof of this theorem is very short and merely based on the inverse function theorem, see Section~\ref{sec:proofs-main-results} below. It applies in particular to power type nonlinearities 
\begin{equation}
  \label{eq:power-type}
f(x,u)=Q(x)|u|^{p-2}u.  
\end{equation}
More precisely, if $p>2$, and $Q \in L^\infty_{\alpha}(\R^N)$ for some $\alpha>\frac{N+1}{2}$, we find that $f(x,\cdot)$ is real differentiable for every $x \in \R^N$, and $f'= \partial_u f \in \cL_{\R}(\C,\C)$ is given by 
$f'(x,u)v = Q(x)\bigl( \frac{p}{2}|u|^{p-2}v + \frac{p-2}{2}|u|^{p-4} u^2 \bar v\bigr)$, which implies that $$
\|f'(x,u)\|_{\cL_{\R}(\C,\C)} \le (p-1)|Q(x)||u|^{p-2} \qquad \text{for $x \in \R^N$, $u \in \C$}. 
$$
From this it is easy to deduce that the assumptions of Theorem~\ref{teo-implicit-function} are satisfied in this case. In particular, for given $\phi \in L^\infty(\R^N)$, Theorem~\ref{teo-implicit-function} yields the existence of $\eps>0$ and a unique local branch $(-\eps,\eps) \to L^\infty(\R^N)$, $\lambda \mapsto u_\lambda$ of solutions of the equation
\begin{equation}\label{nlh-1-integral-parameter-lambda}
u = \cR_k(Q|u|^{p-2}u)+ \lambda \phi \qquad \text{in $L^\infty(\R^N)$.}
\end{equation} 
In our next result, we establish the existence of a global continuation of this local branch. 

\begin{theorem}\label{thm:rabinowitz-applied}
Let $N\geq 3$, $2<p<2^\ast$, $Q\in L^\infty_\alpha(\R^N,\R)\backslash\{0\}$ for some $\alpha > \frac{N+1}{2}$ and $\phi\in L^\infty(\R^N)$. Moreover, let 
$$
\cS_\phi:= \{(\lambda,u)\::\: \lambda \ge 0,\: u \in L^\infty(\R^N),\: \text{$u$ solves (\ref{nlh-1-integral-parameter-lambda})}\}\:\subset \: [0,\infty) \times L^\infty(\R^N),
$$
and let $\cC_\phi \subset \cS_\phi$ denote the connected component of $\cS_\phi$ which contains the point $(0,0)$. 

Then $\cC_\phi \setminus \{(0,0)\}$ is an unbounded subset of $(0,\infty) \times L^\infty(\R^N)$. 
\end{theorem}

\qquad We note that in general the unboundedness of $\cC_\phi$ does not guarantee that $\cC_{\phi}$ intersects $\{1\} \times \R^N$, since the branch given by $\cC_\phi$ may blow up in $L^\infty(\R^N)$ at some value $\lambda \in (0,1)$. In particular, under the general assumptions of Theorem~\ref{thm:rabinowitz-applied}, we cannot guarantee the existence of solutions of the equation (\ref{nlh-1-integral}).  
For this, additional a priori bounds on the set of solutions are needed. We shall find such a priori bounds in the case where $Q \le 0$ in $\R^N$, which is usually refered to as the {\em defocusing case}. Moreover, we require $Q$ to have compact support with some control of its diameter. In the following, we let $L^\infty_c(\R^N)$ denote the set of functions $Q \in L^\infty(\R^N)$ with compact support $\supp Q \subset \R^N$, and we let $L^\infty_c(\R^N, \R)$ denotes the subspace of real-valued functions in $L^\infty_c(\R^N)$. We then have the following result. 

\begin{theorem}\label{thm:unbounded_branch-defocusing}
Let $N\geq 3$, $2<p<2^\ast$, $Q\in L^\infty_c(\R^N,\R)\backslash\{0\}$ and $\phi\in L^\infty(\R^N)$. Assume furthermore that $Q\leq 0$ a.e. in $\R^N$ and $\text{diam}(\text{supp }Q)\le \frac{\zN}{k}$, where  
$\zN$ denotes the first positive zero
of the Bessel function $Y_{\frac{N-2}2}$ of the second kind of order $\frac{N-2}{2}$. 

Then the set $\cC_\phi$ given in Theorem~\ref{thm:rabinowitz-applied} intersects $\{\lambda\} \times L^\infty(\R^N)$ for every $\lambda>0$. In particular, (\ref{nlh-1-integral-parameter-lambda}) admits a solution with $\lambda=1$.
\end{theorem}

 \qquad To put the assumption on the support of $Q$ into perspective, we note that $\zE = \frac{\pi}{2}$ since $Y_{\frac{1}{2}}(t)=- \sqrt{\frac{2}{\pi t}} \cos t$ for $t>0$. Moreover, $\zN > \zE$ for $N > 3$, see \cite[Section 9.5]{AS}. Consequently, the assumptions of Theorem~\ref{thm:unbounded_branch-defocusing} are satisfied if $Q\in L^\infty_c(\R^N,\R)\backslash\{0\}$ is a nonpositive function with $\text{diam}(\text{supp }Q)< \frac{\pi}{2k}$. We also refer to \cite[p. 467]{AS} for a list of the values of $\zN$ for $3 \le N \le 15$.

 \qquad It seems appropriate to compare our results with recent work on the existence of {\em real-valued (standing wave) solutions} of (\ref{nlh-1}). A large class of such real-valued solutions has been detected and studied extensively in recent years by considering the associated integral equation
 \begin{equation}
   \label{eq:real-valued-variational}
u = \Psi_k \ast(N_f(u)),   
 \end{equation}
 where $\Psi_k$ is the real part of the fundamental solution $\Phi_k$, see e.g. \cite{MMP,EW0,EW1,EY} and the references therein. In particular, a variational approach to detect and analyze solutions of (\ref{eq:real-valued-variational}) has been set up in \cite{EW1} for the special case where the nonlinearity $f$ is of the form $f(x,u)=Q(x)|u|^{p-2}u$ with nonnegative $Q \in L^\infty(\R^N,\R)$ and suitable exponents $p>2$. Variants of this variational approach have been developed further in \cite{MMP,EY} under appropriate assumptions on the nonlinearity. However, the variational methods in these papers are of no use in the context of the integral equation  (\ref{nlh-1-integral}) which has no variational structure. The contrast between real standing wave solutions and complex scattering solutions is even more glaring as we shall see that the related homogeneous equation $u = \cR_k [Q|u|^{p-2}u]$ admits only the trivial bounded solution $u\equiv 0$ if $p \ge 2$ and $Q \in L^\infty_\alpha(\R^N,\R)$ for some $\alpha> \frac{N+1}{2}$. Indeed, we shall prove this Liouville type result in Proposition \ref{corol-kato} below by adapting a nonexistence result due to Kato \cite{kato59} to the present nonlinear context.

\qquad In the perturbative setting where a priori smallness assumptions are imposed, the detection of real and complex solutions of (\ref{nlh-1-integral}) follows the same strategy of applying contraction mapping arguments in suitable function spaces. In this context, we mention the paper \cite{Mandel1} where a variant of the contraction mapping argument of Guti\'errez \cite{G} is developed and used to detect continua of small real-valued solutions of (\ref{nlh-1}) for a larger class of nonlinearities than in \cite{G}. More precisely, these continua are found by solving the non-homogeneous variant 
$u = \Psi_k \ast(Q|u|^{p-2}u) +\phi$ of (\ref{eq:real-valued-variational}) for a range of given small real-valued solutions $\phi$ of the homogeneous Helmholtz equation $-\Delta \phi -\phi =0$.

\qquad Due to the lack of a priori smallness assumptions and the lack of a variational structure, our main results given in Theorems 1.2, 1.5 and 1.6 require a different approach than in the above-mentioned papers. As mentioned earlier, this
approach is based on topological fixed point theory, and it therefore requires suitable a priori bounds. With regard to this aspect, the present paper is related to \cite{EW2} where continuous branches of real-valued standing wave solutions of (\ref{eq:real-valued-variational}) have been constructed. However, while the derivation of suitable priori bounds is the key step both in \cite{EW2} and in the present paper, these bounds are of different nature as they relate to different integral equations and to different classes of solutions. In \cite{EW2}, under suitable additional assumptions on $Q$ and $p$, a priori bounds are derived for real-valued solutions of $u = \Psi_k \ast(Q|u|^{p-2}u)$ which are positive within the support of the nonlinearity $f$. In contrast, here we need a priori bounds for complex solutions of (\ref{nlh-1-integral}), and for this we cannot use positivity properties and local maximum principles. Instead, the approach of the present paper is based on a Liouville theorem relying on Sommerfeld's radiation condition and on combining regularity and test function estimates with local monotonicity properties of the function $\Psi_k$, see Sections \ref{nonexistence} and~\ref{sec:priori-estim-defoc} below.

\qquad The paper is organized as follows.   In Section~\ref{sec:estim-helmh-resolv} we establish basic estimates of the resolvent operator $\cR_k$, and we prove Proposition~\ref{resolvent-compact-and-continuous}. In Section~\ref{sec:estim-subst-oper}, we show useful estimates and regularity properties of the substitution operator associated with the nonlinearity $f(x,u)$.   In order to apply topological fixed point theory,  we first need to prove the nonexistence of solutions to linear and superlinear integral equations related to the operator $\cR_k$. This will be done in Section~\ref{nonexistence}. In Section~\ref{sec:priori-estim-defoc}, we then prove a priori bounds for solution of equation \eqref{nlh-1-integral} and related variants under various assumptions on the nonlinearity $f$. The proof of the main theorems is then completed in Section~\ref{sec:proofs-main-results}. Finally, in the appendix, we provide a relative a priori bound based on bootstrap regularity estimates between $L^p$-spaces which is used in the proof of Theorem~\ref{thm:unbounded_branch-defocusing}.

\section{Estimates for the Helmholtz resolvent operator}
\label{sec:estim-helmh-resolv}

\begin{lemma}\label{lm 2.1}
Let $N\geq2$, $k>0$ and for $\alpha>\frac{N+1}{2}$, let $\tau(\alpha)$ be defined by (\ref{exp 1}). 
 Then for any $ v\in L^\infty_{\alpha } (\RR^N)$   and $\alpha>\frac{N+1}{2}$, we have
\[
\||\Phi_k|* v\|_{L^\infty_{\tau(\alpha)}}\leq C\| v\|_{L^\infty_{\alpha }},\quad \||\nabla \Phi_k|* v\|_{L^\infty_{\tau(\alpha)}}\leq C\| v\|_{L^\infty_{\alpha }},
\]
where the constant $C>0$ depends only on $N$, $\alpha$ and $k$.
\end{lemma}

\begin{proof}
In the following, the letter $C>0$ always denotes constants which only depends on $N$, $\alpha$ and $k$. We observe that 
\[
|\Phi_k(x)|\leq
\begin{cases}
C\, |x|^{2-N}\quad &\text{if }\ N\geq 3,\\[1.5mm]
C\, \log \frac2{|x|}   &\text{if }\ N=2,\,
\end{cases}\qquad |\nabla\Phi_k|\leq c|x|^{1-N}
\quad\,\text{for $0<|x|\leq 1$} \]
and
\[ |\Phi_k(x)|,\, |\nabla\Phi_k|\leq C\, |x|^{\frac{1-N}2} \quad \text{if } |x|>1. \]
It then follows that
\begin{align*}
  &|(|\Phi_k|* v)(x)|\leq \int_{\RR^N} |\Phi_k(z)|\,| v(x-z)|\, dz\notag\\
            &\leq \begin{cases}
            C\| v\|_{L^\infty_\alpha}\,\Big(\int_{B_1(0)}|z|^{2-N}\langle
              x-z\rangle^{-\alpha}\, dz + \int_{\RR^N\backslash B_1(0)}|z|^{\frac{1-N}2}\, \langle x-z\rangle^{-\alpha}\, dz\Big)\ \ \text{if }\ N\geq3 ,\\[3mm]
           C\| v\|_{L^\infty_\alpha}\,\Big(\int_{B_1(0)}\log\frac2{|z|} \langle
              x-z\rangle^{-\alpha}\, dz + \int_{\RR^N\backslash B_1(0)}|z|^{\frac{1-N}2}\, \langle x-z\rangle^{-\alpha}\, dz\Big)\ \ \text{if }\ N=2.
              \end{cases}
\end{align*}

 For $|x|\leq 4$, it is easy to see that
\begin{align}\label{2.2}
  |(|\Phi_k|* v)(x)|\leq \begin{cases}
            C\| v\|_{L^\infty_\alpha}\,\Big(\int_{B_1(0)}|z|^{2-N} \,dz + \int_{\RR^N\backslash B_1(0)}|z|^{\frac{1-N}2-\alpha}\,  dz\Big)\ \ \text{if }N\geq3 \\[3mm]
           C\| v\|_{L^\infty_\alpha}\,\Big(\int_{B_1(0)}\log\frac2{|z|}  \, dz + \int_{\RR^N\backslash B_1(0)}|z|^{\frac{1-N}2-\alpha}\, dz\Big)\ \ \text{if }N=2,
\end{cases}  \end{align}
and 
\begin{align}\label{2.2-gradient}
  |(|\nabla\Phi_k|* v)(x)|\leq   C\| v\|_{L^\infty_\alpha}\,\Big(\int_{B_1(0)}|z|^{1-N} \,dz + \int_{\RR^N\backslash B_1(0)}|z|^{\frac{1-N}2-\alpha}\,  dz\Big),
   \end{align}
where $\frac{1-N}2-\alpha<-N$.

\qquad In the following, we consider $|x|>4$.  Since $\alpha>\frac{N+1}{2}$, direct computation shows that 
\begin{align*}
I_1&:=\begin{cases} \int_{B_1(0)}|z|^{2-N}\, \langle
              x-z\rangle^{-\alpha}\, dz\quad{\rm if}\ N\geq 3 \\[1mm]
  \int_{B_1(0)}\log\frac{2}{|z|}\, \langle
              x-z\rangle^{-\alpha}\, dz\quad{\rm if}\ N=2
              \end{cases}
               \\[1.5mm]
             & \displaystyle\,\leq C |x|^{-\alpha} \le C \langle x\rangle^{-\alpha}.
 \end{align*}
Moreover, 
   \begin{align*}  
I_2&:=\int_{B_{\frac{|x|}2} (0)\setminus B_1(0)} |z|^{\frac{1-N}2}\, \langle
     x-z\rangle^{-\alpha}\, dz  \leq C|x|^{-\alpha} \int_{B_{\frac{|x|}2} (0)\setminus B_1(0)} |z|^{\frac{1-N}2} dz
    \leq C|x|^{-\alpha+\frac{N+1}{2}},\\
I_3&:=\int_{B_{\frac{|x|}2} (x)} |z|^{\frac{1-N}2}\, \langle
     x-z\rangle^{-\alpha}\, dz  \leq C|x|^{-\frac{N-1}{2}} \int_{B_{\frac{|x|}2} (x) } \langle
     x-z\rangle^{-\alpha} dz
    \leq C|x|^{-\tau(\alpha)}
 \end{align*}
 and
 \begin{align*}  I_4:&=\int_{\RR^N\setminus(B_{\frac{|x|}2} (0)\cup B_{\frac{|x|}2} (x)) } |z|^{\frac{1-N}2}\, \langle
     x-z\rangle^{-\alpha}\, dz \\
    &= |x|^{-\alpha+\frac{N+1}{2}} \int_{\RR^N\setminus(B_{\frac{1}2} (0)\cup B_{\frac{1}2} (e_x)) }  |z|^{\frac{N-1}{2}}|z-\hat x|^{-\alpha} dz
    \leq C|x|^{-\alpha+\frac{N+1}{2}},
 \end{align*}
 where $\hat x=\frac{x}{|x|}$. Since $-\tau(\alpha)\geq \max\{-\frac{N-1}{2}, -\alpha, -\alpha+\frac{N+1}{2}\}$, we may combine these estimates with (\ref{2.2}) to see that
\[|(|\Phi_k|*v)(x)|\leq C\| v\|_{L^\infty_{ \alpha}}\,\bigg(\sum^{4}_{j=1}I_j\bigg) \leq C \langle x \rangle^{-\tau(\alpha)} \| v\|_{L^\infty_{ \alpha}} \quad  \text{for all $x\in  \RR^N.$}\]
Moreover, noting that 
\begin{align*}
\tilde I_1:=  \int_{B_1(0)}|z|^{1-N}\, \langle
              x-z\rangle^{-\alpha}\, dz 
                 \leq C |x|^{-\alpha} \leq C \langle x\rangle^{-\alpha}\qquad \text{for $|x|>4$,}
 \end{align*}
we find by (\ref{2.2-gradient}) that  
 \[|(|\nabla \Phi_k|*v)(x)|\leq C\| v\|_{L^\infty_{ \alpha}}\,\bigg(\tilde I_1+\sum^{4}_{j=2}I_j\bigg) \leq C \langle x \rangle^{-\tau(\alpha)}\| v\|_{L^\infty_{ \alpha}}\quad  \text{for all $x\in  \RR^N.$}\]
 The proof is thus complete.
\end{proof}

\begin{proof}[Proof of Proposition~\ref{resolvent-compact-and-continuous}]
(i) Clearly, Lemma~\ref{lm 2.1} yields (\ref{eq:kappa-sigma-finite}) and therefore the continuity of the linear resolvent operator $\cR_k: L^\infty_\alpha(\RR^N) \to L^\infty_{\tau(\alpha)}(\RR^N)$, whereas the latter space is continuously embedded in $L^\infty(\RR^N)$. To see the compactness of $\cR_k$ as a map $L^\infty_\alpha(\RR^N) \to L^\infty(\RR^N)$, let $(u_n)_n$ be a sequence in 
$L^\infty_\alpha(\RR^N)$ with 
$$
m:= \sup_{n \in \N} \|u_n \|_{L^\infty_{\alpha}}< \infty.
$$

Moreover, let $v_n:= \cR_k u_n = \Phi * u_n$ for $n \in \N$. By Lemma~\ref{lm 2.1}, we then have 
\begin{equation}
  \label{eq:proof-est-compactness}
\|v_n\|_{L^\infty_{\tau(\alpha)}} \le C m \quad \text{and}\qquad \|\nabla v_n\|_{L^\infty_{\tau(\alpha)}}= 
\|\nabla \Phi * u_n\|_{L^\infty_{\tau(\alpha)}} \le C m
\end{equation}
for all $n \in \N$. In particular, the sequence $(v_n)_n$ is bounded in $C^1_{loc}(\R^N)$. By the Arzel\`a-Ascoli theorem, there exists $v \in L^\infty_{loc}(\R^N)$ with 
\begin{equation}
  \label{eq:proof-locally-uniformly}
\text{$v_n \mapsto v$ locally uniformly on $\R^N$.}  
\end{equation}
By (\ref{eq:proof-est-compactness}), it then follows that $v \in L^\infty_{\tau(\alpha)}(\R^N)$ with $\|v\|_{L^\infty_{\tau(\alpha)}} \le C m$. 

\qquad Moreover, for given $R>0$ we have, with $A_R:= \R^N \setminus B_R(0)$ 
$$
\|v_n-v\|_{L^\infty(A_R)}\le \|v_n\|_{L^\infty(A_R)} + \|v\|_{L^\infty(A_R)} \le R^{-\tau(\alpha)}\Bigl(\|v_n\|_{L^\infty_{\tau(\alpha)}} + \|v\|_{L^\infty_{\tau(\alpha)}}\Bigr) \le 2Cm R^{-\tau(\alpha)}.
$$
Combining this estimate with (\ref{eq:proof-locally-uniformly}), we see that $\limsup \limits_{n \to \infty}\|v_n-v\|_{L^\infty(\R^N)}\le 
2Cm R^{-\tau(\alpha)}$ for every $R>0$. Since $\tau(\alpha)>0$, we conclude that $v_n \to v$ in $L^\infty(\R^N)$. This shows the compactness of the operator $L^\infty_\alpha(\RR^N) \to L^\infty(\RR^N)$.\\
(ii) Let $\alpha > \frac{N(N+3)}{2(N+1)}$ and $h \in L^\infty_\alpha(\R^N)$. It then follows that $h \in L^{\frac{2(N+1)}{N+3}}(\R^N)$. Consequently, \cite[Proposition A.1]{EW1} implies that $u= \cR_k h$ is a strong solution of $-\Delta u - k^2 u = h$. Moreover, $u$ satisfies~(\ref{eqn:sommerfeld1-averaged}) by the estimate in \cite[Theorem 8]{G} and the remark following it.

Finally, we suppose that $\alpha > N$. In this case, the linear map
$$
\widetilde\cR_k: L^\infty_\alpha(\R^N) \to L^\infty_{\text{\tiny $\frac{N-1}{2}$}}(\R^N), \qquad v \mapsto \widetilde\cR_k(v):=\frac{d \cR_k v}{dr} -ik \cR_k v
$$
is well-defined and bounded by Lemma~\ref{lm 2.1}. Moreover, if $h \in L^\infty(\R^N)$ has compact support, the fact that $\Phi_k$ satisfies~(\ref{sommerfeld-1}) and elementary convolution estimates show that $u= \cR_k h$ also satisfies~(\ref{sommerfeld-1}). In the general case $h \in L^\infty_\alpha(\R^N)$, we consider a sequence of functions $h_n \in L^\infty_\alpha(\R^N)$ with compact support and such that $h_n \to h$ in $L^\infty_\alpha(\R^N)$, which then also implies that
\begin{equation}
  \label{eq:extra-argument-strong-s-c}
\widetilde\cR_k h_n \to \widetilde\cR_k h \qquad \text{in $\: L^\infty_{\text{\tiny $\frac{N-1}{2}$}}(\R^N).$}  
\end{equation}
Moreover, for every $n \in \N$ we have 
\begin{align*}
\limsup_{|x| \to \infty} |x|^{\frac{N-1}{2}}\bigl|[\widetilde\cR_k h](x)\bigr| &\le
                                                                              \limsup_{|x| \to \infty} |x|^{\frac{N-1}{2}}\bigl|[\widetilde\cR_k h_n](x)\bigr| + \|\widetilde\cR_k h-\widetilde\cR_k h_n\|_{L^\infty_{\text{\tiny $\frac{N-1}{2}$}}}\\
&= \|\widetilde\cR_k h-\widetilde\cR_k h_n\|_{L^\infty_{\text{\tiny $\frac{N-1}{2}$}}},    
\end{align*}
and thus
$$
\limsup_{|x| \to \infty} |x|^{\frac{N-1}{2}}\bigl|[\widetilde\cR_k h](x)\bigr| \le \lim_{n \to \infty}\|\widetilde\cR_k h-\widetilde\cR_k h_n\|_{L^\infty_{\text{\tiny $\frac{N-1}{2}$}}} = 0
$$
by (\ref{eq:extra-argument-strong-s-c}). Hence $u= \cR_k h$ satisfies~(\ref{sommerfeld-1}).
\end{proof}

\section{Estimates for the substitution operator}
\label{sec:estim-subst-oper}

\begin{lemma}
\label{lem:nemytskii_cont} 
Let, for some $\alpha \in \R$, the nonlinearity $f: \R^N \times \C \to \C$ be a continuous function satisfying 
\begin{equation}
  \label{eq:assumption-f1-1}
S_{f,M,\alpha}:=  \sup_{|u|\le M,x \in \R^N} \langle x \rangle^{\alpha}|f(x,u)|< \infty \qquad \text{for all $M>0$.}   
\end{equation}
Then the superposition operator 
$$
N_f: L^\infty (\R^N) \to L^{\infty}_{\alpha'}(\R^N),\qquad N_f(u)(x):= f(x,u(x))
$$
is well defined, bounded and continuous for every $\alpha'<\alpha$.
\end{lemma}

\begin{proof}
It clearly follows from (\ref{eq:assumption-f1-1}) that $N_f$ is well defined and satisfies the estimate 
$$
\|N_f(u)\|_{L^\infty_{\alpha'}}\le \|N_f(u)\|_{L^\infty_\alpha} \le S_{f,M,\alpha} \qquad \text{for $M>0$ and $u \in L^\infty(\R^N)$ with $\|u\|_{L^\infty} \le M$.}
$$
To see the continuity we consider a sequence $(u_n)_n \subset L^\infty(\R^N)$ with $u_n \to u$ in $L^\infty(\R^N)$, and we put 
$M:= \sup \{\|u_n\|_{L^\infty}\::\: n \in \N\}$. For given $R>0$ we have, with $B_R:= B_R(0)$ and 
$A_R:= \R^N \setminus B_R$, 
\begin{align*}
\|N_f(u_n)-N_f(u)\|_{L^\infty_{\alpha'}(A_R)} & \le \|N_f(u_n)\|_{L^\infty_{\alpha'}(A_R)} +\|N_f(u)\|_{L^\infty_{\alpha'}(A_R)}\\
&\le R^{\alpha'-\alpha}\Bigl(
\|N_f(u_n)\|_{L^\infty_{\alpha}(A_R)} +\|N_f(u)\|_{L^\infty_{\alpha}(A_R)}\Bigr)\\ 
&\le  2 S_{f,M,\alpha} R^{\alpha'-\alpha}.
\end{align*}
Moreover, since $f$ is uniformly continuous on $D_R:= \{(x,z) \in \R^N \times \C\::\: \|x\| \le R,\:|z| \le M\}$, we find that 
$$
\|N_f(u_n)-N_f(u)\|_{L^\infty(B_R)}= \sup_{|x| \le R}|f(x,u_n(x))-f(x,u(x))| \to 0 \qquad \text{as $n \to \infty$.}
$$

  We thus infer that $\limsup \limits_{n \to \infty}\|N_f(u_n)-N_f(u)\|_{L^\infty_{\alpha'}(\R^N)}\le 2 S_{f,M,\alpha} R^{\alpha'-\alpha}$ for every $R>0$. Since $\alpha' < \alpha$ by assumption, we conclude that $N_f(u_n) \to N_f(u)$ in $L^\infty_{\alpha'}(\R^N)$. This shows the continuity of $N_f: L^\infty(\RR^N) \to L^\infty_{\alpha'}(\R^N)$.  
\end{proof}

\begin{lemma}\label{lem:nemytskii_C-1} 
Let, for some $\alpha>\frac{N+1}{2}$, the nonlinearity $f: \R^N \times \C \to \C$ be a continuous function satisfying (\ref{eq:assumption-f1-1}). Suppose moreover that the function $f(x,\cdot):\C \to \C$ is real differentiable for every $x \in \R^N$, and that $f':= \partial_u f: \R^N \times \C \to \cL_{\R}(\C,\C)$ is a continuous function satisfying 
\begin{equation}
  \label{eq:assumption-f1-1-C-1}
T_{f,M,\alpha}:=  \sup_{|u|\le M,x \in \R^N} \langle x \rangle^{\alpha}\|f'(x,u)\|_{\cL_\R(\C,\C)}< \infty \qquad \text{for all $M>0$.}   
\end{equation}
Then the superposition operator $N_f: L^\infty (\R^N) \to L^{\infty}_{\alpha'}(\R^N)$ is of class $C^1$ for $\alpha' < \alpha$ with 
\begin{equation}
  \label{eq:expression-derivative}
N_f'(u) := N_{f'}(u) \qquad  \text{for $u \in L^\infty(\R^N)$,}
\end{equation}
where $N_{f'}(u) \in \cL_\R(L^\infty(\R^N),L^\infty_{\alpha'}(\R^N))$ is defined by 
\begin{equation}
  \label{eq:definition-substitution-deriv}
[N_f'(u)v](x):= f'(x,u(x))v(x) \qquad \text{for $v \in L^\infty(\R^N), x \in \R^N$.}
\end{equation}
\end{lemma}

\begin{proof}
For the sake of brevity, we put $X:= L^\infty(\R^N)$ and $Y:= L^{\infty}_{\alpha'}(\R^N)$. By assumption (\ref{eq:assumption-f1-1-C-1}) and a very similar argument as in the proof of Lemma~\ref{lem:nemytskii_cont}, the nonlinear operator 
$$
N_{f'}: X \to \cL_\R(X,Y)
$$
defined by (\ref{eq:definition-substitution-deriv}) is well-defined, bounded and continuous. Thus, it suffices to show that $N_f$ is G\^ateaux-differentiable, and that (\ref{eq:expression-derivative}) is valid as a directional derivative. So let $u,v \in X$, and let $M:= \|u\|_{L^\infty}+ \|v\|_{L^\infty}$. For $\theta \in \R$ and $x \in \R^N$, we estimate 
\begin{align*}
\Bigl|&\frac{N_f(u+\theta v)(x)-N_f(u)(x)}{\theta}-[N_{f'}(u)v](x)\Bigr|\\
&= \Bigl|\frac{f(x,[u+\theta v](x))-f(x,u(x))}{\theta}-f'(x,u(x))v(x)\Bigr|\\
&= \Bigl|\int_0^1 \bigl[f'(x,[u+\xi \theta v](x))-f'(x,u(x))\bigr] v(x)\, d\xi \Bigr| \le |v(x)| g_\theta(x)
\end{align*}
with 
$$
g_\theta(x):= \sup_{\xi\in[0,1]}\bigl\|f'(x,[u+\xi \theta v](x))-f'(x,u(x))\bigr\|_{\cL_{\R}(\C,\C)} \qquad \text{for $\theta \in \R$, $x \in \R^N$.}
$$

Since $\|u + \tau v\|_{L^\infty} \le M$ for $\tau \in \R$, $|\tau| \le 1$, we have
$$
|g_\theta(x)|\le \sup_{\tau \in[0,1]}\|f'(x,[u+\tau v](x))\|_{\cL_{\R}(\C,\C)} 
+ \|f'(x,u(x))\|_{\cL_{\R}(\C,\C)} \le 2T_{f,M,\alpha}\langle x\rangle^{-\alpha} 
$$
for $|\theta| \le 1$, $x \in \R^N$. Similarly as in the proof of Lemma~\ref{lem:nemytskii_cont}, we now define, for given $R>0$, $B_R:= B_R(0)$, $A_R:= \R^N \setminus B_R$, and $D_R:= \{(x,z) \in \R^N \times \C\::\: \|x\| \le R,\:|z| \le M\}$. From the estimate above, it then follows 
\begin{equation}
\label{diff-substitution-est-1}
\Bigl \| \frac{N_f(u+\theta v)-N_f(u)}{\theta}-N_{f'}(u)v \Bigr\|_{L^\infty_{\alpha'}(A_R)}  \le 2 \|v\|_X  T_{f,M,\alpha}R^{\alpha'-\alpha}.
\end{equation}
Moreover, since, by assumption, $f'$ is uniformly continuous on the compact set $D_R$, we find that 
$$
\|g_\theta\|_{L^\infty(B_R)} \to 0  \qquad \text{as $\theta \to 0$}.
$$
We thus conclude that 
$$
\limsup \limits_{\theta \to 0}\, \Bigl \|\frac{N_f(u+\theta v)-N_f(u)}{\theta}-M(u)v \Bigr \|_{L^\infty_{\alpha'}(\R^N)} \le 2 \|v\|_X\,  T_{f,M,\alpha}\,R^{\alpha-\alpha'}\qquad \text{for every $R>0$.}
$$
Since $\alpha' < \alpha$ by assumption, we conclude that $\frac{N_f(u+\theta v)-N_f(u)}{\theta} \to N_{f'}(u)v$ in $Y$ as $\theta \to 0$. The proof is thus finished.
\end{proof}

\section{Nonexistence of outgoing waves for the nonlinear Helmholtz equation}
\label{nonexistence}

To begin this section, we recall the following nonexistence result for eigenfunctions of Schr\"odinger operators with positive eigenvalue. 
It is a consequence of a result by Alsholm and Schmidt \cite[Proposition 2 of Appendix 3]{alsholm-schmidt70} extending earlier results due to Kato \cite{kato59}:

\begin{proposition}[see {\cite[Proposition 2]{alsholm-schmidt70}}]\label{prop:kato}
Let $u\in W^{2,2}_{\text{loc}}(\R^N,\C)$ solve $-\Delta u +Vu=k^2u$ in $\R^N$, where $V\in L^\infty(\R^N)$
satisfies 
\begin{equation}
\label{condition-V}
|V(x)|\leq C\langle x\rangle^{-1-\eps} \qquad \text{for a.e. $x\in\R^N$ with constants $C, \eps>0$.}
\end{equation}
If 
$$
\liminf_{R\to\infty}\frac1R\int_{B_R(0)}(|\nabla u|^2+k^2|u|^2)\, dx=0,
$$
then there exists $R>0$ such that $u$ vanishes identically in $\R^N\backslash B_R(0)$ for some $R>0$.\\
If, moreover, $V$ is real-valued, then $u$ vanishes identically in $\R^N$.
\end{proposition}

\begin{proof}
It has been proved in {\cite[Proposition 2]{alsholm-schmidt70}} that $u$ vanishes identically in $\R^N\backslash B_R(0)$ for some $R>0$.
Assuming in addition that $V$ is real-valued, we then deduce by a unique continuation result that $u$ vanishes identically on $\R^N$. 
More precisely, for $u_1=\text{Re}(u)$ and $u_2=\text{Im}(u)$ we have 
$|\Delta u_i|\leq C |u_i|$ on $\R^N$ with some constant $C>0$. The strong unique continuation property \cite[Theorem 6.3]{jerison-kenig85} 
(see also Remark 6.7 in the same paper) therefore implies $u_1=u_2=0$ on $\R^N$, and this concludes the proof.  
\end{proof}

\qquad From Proposition~\ref{prop:kato}, we shall now deduce the following nonexistence result for linear and superlinear variants of the corresponding integral equation involving the Helmholtz resolvent operator. 

\begin{proposition}
\label{corol-kato}
Let $N \ge 3$, $2 \le p < \infty$, $\alpha >\frac{N+1}{2}$, and let $u \in L^\infty(\R^N)$ be a solution of 
\begin{equation}
  \label{eq:corol-kato-eq}
u = \cR_k [Q|u|^{p-2}u]
\end{equation}
with a function $Q \in L^\infty_\alpha(\R^N, \R)$. Then $u \equiv 0$.   
\end{proposition}

\begin{proof}
Let $V:= Q|u|^{p-2}$, so that (\ref{eq:corol-kato-eq}) writes in the form
  \begin{equation}
  \label{eq:corol-kato-eq-variant}
u = \cR_k [V u]
\end{equation}
We then have $V \in L^\infty_\alpha(\R^N, \R)$ and also $Vu \in L^\infty_\alpha(\R^N)$ since  $u \in L^\infty(\R^N)$. Therefore Proposition~\ref{resolvent-compact-and-continuous} implies that $u \in L^\infty_{\tau(\alpha)}(\R^N)$ with $\tau(\alpha)$ given in (\ref{exp 1}). It then follows that $Vu \in L^\infty_{\alpha_1}(\R^N)$
with $\alpha_1 = \alpha +\tau(\alpha)$ and hence $u \in L^\infty_{\tau(\alpha_1)}(\R^N)$ again by Proposition~\ref{resolvent-compact-and-continuous}. Defining inductively $\alpha_{k}:= \alpha_{k-1}+\tau(\alpha_{k-1})$ for $k \ge 2$, we may iterate the application of Proposition~\ref{resolvent-compact-and-continuous} to obtain that $u \in L^\infty_{\tau(\alpha_k)}(\R^N)$ for all $k \in \N$. After a finite number of steps, we therefore deduce from (\ref{exp 1}) that
$u \in L^\infty_{\text{\tiny $\frac{N-1}{2}$}}(\R^N)$ and therefore $Vu \in L^\infty_{\alpha+\text{\tiny $\frac{N-1}{2}$}}(\R^N)$. Since $\alpha > \frac{N+1}{2}$ by assumption, this implies that $Vu \in L^\infty(\R^N) \cap L^1(\R^N)$. It then follows e.g. from \cite[Proposition A.1]{EW1} that $u \in W^{r}_{\text{loc}}(\R^N) \cap L^{\frac{2(N+1)}{N-1}}(\R^N) \cap L^\infty(\R^N)$ for $r < \infty$, and $u$ is a strong solution of the differential equation
\begin{equation}\label{eqn:nlh_power}
-\Delta u  -k^2 u =  V u\quad\text{in }\R^N.
\end{equation}
Moreover, by \cite[Theorem 8]{G} and the remark following it, $u$ satisfies the Sommerfeld outgoing radiation condition in the form
given in (\ref{eqn:sommerfeld1-averaged}), e.g. 
\begin{equation}
  \label{eq:sommerfeld-proof}
\lim_{R\to\infty}\frac{1}{R} \int_{B_R}\left|\nabla
u(x)-iku(x)\frac{x}{|x|} \right|^2\, dx=0.
\end{equation}
We now proceed similarly as in the proof of Corollary 1 in \cite{G}. Expanding the terms in (\ref{eq:sommerfeld-proof}), the condition can be rewritten as
\begin{align}\label{eqn:som2}
\lim_{R\to\infty}\frac1{R}\left\{\int_{B_R}(|\nabla u|^2+k^2|u|^2)\, dx 
- 2k \int_0^R \text{Im}\left(\int_{\partial B_\rho}\overline{u}\nabla u\cdot\frac{x}{|x|}\, d\sigma\right)\, d\rho\right\}=0.
\end{align}
Since $u\in W^{2,2}_{\text{loc}}(\R^N)$ solves \eqref{eqn:nlh_power} in the strong sense, the divergence theorem
gives
\begin{align*}
\int_{\partial B_\rho}\overline{u}\nabla u\cdot\frac{x}{|x|}\, d\sigma
&=\int_{B_\rho}|\nabla u|^2\, dx + \int_{B_\rho}\overline{u}\Delta u\, dx\\
&=\int_{B_\rho}|\nabla u|^2\, dx - \int_{B_\rho} (k^2|u|^2 + V|u|^2)\, dx,
\end{align*}
where the right-hand side in the last line is purely real-valued, since by assumption $V= Q|u|^{p-2}$ takes only real values.
Consequently, we find
$$
 \text{Im}\left( \int_{\partial B_\rho}\overline{u}\nabla u\cdot\frac{x}{|x|}\, d\sigma \right)=0
$$
for all $\rho>0$, and plugging this into \eqref{eqn:som2} yields
\begin{equation}\label{eqn:som3}
\lim_{R\to\infty}\frac1{R}\int_{B_R}(|\nabla u|^2+k^2|u|^2)\, dx=0.
\end{equation}
Moreover, since $V \in L^\infty_\alpha(\R^N)$ and $\alpha> \frac{N+1}{2}>1$, condition (\ref{condition-V}) is satisfied for $V$. Hence Proposition~\ref{prop:kato} implies that $u \equiv 0$ on $\R^N$.
\end{proof}

\section{A priori bounds for solutions}
\label{sec:priori-estim-defoc}

The aim of this section is to collect various a priori bounds for solutions of (\ref{nlh-1-integral}) under different assumptions on the nonlinearity $f$.

\subsection{A priori bounds for the case of linearly bounded nonlinearities} 
\label{sec-a-priori-linearly-bounded}

In this subsection we focus on linearly bounded nonlinearities, and we prove the following boundedness property.

\begin{proposition}
  \label{sec:proof-theorem-refw-1}
Let, for some $\alpha>\frac{N+1}{2}$, the nonlinearity $f$ satisfy the assumption
\begin{equation}
  \label{eq:assumption-f1-section}
 \sup_{|u|\le M,x \in \R^N}\langle x\rangle^{\alpha}|f(x,u)|< \infty \qquad \text{for all $M>0$}   
\end{equation}
and \underline{one} of the assumptions $(f_1)$ or $(f_2)$ from Theorem~\ref{W teo 1-sublinear}.

Moreover, let $\phi \in L^\infty(\R^N)$, and let $\cF \subset L^\infty(\R^N)$ be the set of functions $u$ which solve the equation 
\begin{equation}
\label{schaefer-equation}
u = \mu \Bigl(\cR_k N_f(u) + \phi\Bigr) \qquad \text{for some $\mu \in [0,1]$.}
\end{equation}
Then $\cF$ is bounded in $L^\infty(\R^N)$.
\end{proposition}

\begin{proof}
We first assume $(f_2)$. Let $u \in \cF$. By (\ref{schaefer-equation}) and Proposition~\ref{resolvent-compact-and-continuous}, we then have 
\begin{align*}
\|u\|_{L^\infty} &\le \|\cR_k N_f(u)\|_{L^\infty} + \|\phi\|_{L^\infty} \le \bigl\| |\Phi| * N_f(u) \bigr\|_{L^\infty_{\tau(\alpha)}}
+ \|\phi\|_{L^\infty}\\
&\le \kappa_{\alpha} \|N_f(u)\|_{L^\infty_{\alpha}}
+ \|\phi\|_{L^\infty} \le \kappa_{\alpha} \Bigl(\|Q |u|\|_{L^\infty_{\alpha}} + \|b\|_{L^\infty_{\alpha}}\Bigr)+ \|\phi\|_{L^\infty}\\
&\le \kappa_{\alpha} \|Q\|_{L^\infty_{\alpha}}\|u\|_{L^\infty}  + \kappa_{\alpha} \|b\|_{L^\infty_{\alpha}} + \|\phi\|_{L^\infty}.
\end{align*}
Since $\kappa_{\alpha} \|Q\|_{L^\infty_{\alpha}}<1$ by assumption, we conclude that 
$$
\|u\|_{L^\infty} \le \bigl( 1- \kappa_{\alpha} \|Q\|_{L^\infty_{\alpha}}\bigr)^{-1}\bigl(\kappa_{\alpha} \|b\|_{L^\infty_{\alpha}} + \|\phi\|_{L^\infty}\bigr), 
$$
and this shows the boundedness of $\cF$.  

\qquad Next we assume $(f_1)$. In this case we argue by contradiction, so we assume that there exists a sequence $(u_n)_n$ in $\cF$ such that $c_n:= \|u_n\|_{L^\infty} \to \infty$ as $n \to \infty$. Moreover, we let $\mu_n \in [0,1]$ be such that (\ref{schaefer-equation}) holds with $u=u_n$ and $\mu= \mu_n$. We then define $w_n:= \frac{u_n}{c_n} \in L^\infty(\R^N)$, so that $\|w_n\|_{L^\infty}= 1$ and, by assumption $(f_1)$, 
\begin{equation}
\label{schaefer-equation-wn-proof}
w_n = \mu_n \cR_k (a w_n + g_n) + \frac{\mu_n}{c_n} \phi  \qquad \text{with $g_n \in L^\infty_\alpha(\R^N)$, $g_n(x)= \frac{b(x,c_n w_n(x))}{c_n}$.}
\end{equation}
Passing to a subsequence, we may assume that $\mu_n \to \mu \in [0,1]$. Moreover, by assumption $(f_1)$ we have 
$$
g_n \to 0 \qquad \text{in $L^\infty_\alpha(\R^N)$ as $n \to \infty$,}
$$
whereas the sequence $(a w_n)_n$ is bounded in $L^\infty_\alpha(\R^N)$. Since also $\frac{\mu_n}{c_n} \to 0$ as $n \to \infty$, it follows from the compactness of the operator $\cR_k:L^\infty_\alpha(\R^N) \to L^\infty(\R^N)$ that, after passing to a subsequence, $w_n \to w \in L^\infty(\R^N)$. From this we then deduce that 
$$
a w_n \to a w \qquad \text{in $L^\infty_\alpha(\R^N)$,} 
$$
and passing to the limit in (\ref{schaefer-equation-wn-proof}) yields 
$$
w = \mu \cR_k [a w]= \cR_k [\mu a w].
$$
Applying Proposition~\ref{corol-kato} with $p=2$ and $Q:= \mu a$, we conclude that $w \equiv 0$, but this contradicts the fact that $\|w\|_\infty = \lim \limits_{n \to \infty}\|w_n\|_\infty = 1$. Again, we infer the boundedness of $\cF$ in $L^\infty(\R^N)$.
\end{proof}

\subsection{A priori bounds in the superlinear and defocusing case}
\label{sec:priori-bounds-superl}

In this subsection we restrict our attention to the case $f(x,u)= Q(x)|u|^{p-2}u$ with $Q \le 0$. In this case, we shall prove the following a priori estimate. 

\begin{proposition}\label{prop:apriori_defocusing}
Let $N\geq 3$, $k>0$, $2<p<2^\ast$, $Q\in L^\infty_c(\R^N,\R)\backslash\{0\}$ and $\varphi\in L^\infty(\R^N)$.
Assume that 
\begin{itemize}
\item[(Q1)] $Q\leq 0$ a.e $\R^N$ and
\item[(Q2)] $\text{diam}(\text{supp }Q)\le \frac{\zN}{k}$, 
where $\zN$ denotes the first positive zero
of the Bessel function $Y_{\frac{N-2}2}$ of the second kind of order $\frac{N-2}{2}$.
\end{itemize}
Then, there exist $C=C(N,k,p,\|Q\|_\infty, |\text{supp }Q|)>0$ 
and $m=m(N,k,p)\in\N$
such that for any solution $u\in L^\infty(\R^N)$ of
\begin{equation}\label{eqn:fp_complex}
\begin{aligned}
u=\cR_k\bigl(Q |u|^{p-2}u\bigr) +\varphi
\end{aligned}
\end{equation}
we have
\begin{equation}
  \label{eq:apriori-superlinear-defocusing-estimate}
\|u\|_\infty\leq C\left(1+\|\varphi\|_\infty^{(p-1)^m}\right).
\end{equation}
\end{proposition}

\qquad For the proof, we first need two preliminary lemmas. The first lemma gives a sufficient condition for the nonnegativity
of the Fourier transform of a radial function. It is well known in the case $N=3$ (see for example \cite{tuck06}). Since we could
not find any reference for the general case, we give a proof for completeness.

\begin{lemma}\label{lem:FT_rad_positive}
Let $N\geq 3$ and consider $f\in L^1(\R^N)$ radially symmetric, i.e., $f(x)=f(|x|)$, such that $f\geq 0$ on $\R^N$.
If the function $t\mapsto t^{\frac{N-1}2}f(t)$ is nonincreasing on $(0,\infty)$, then $\wh{f}\geq 0$ on $\R^N$.
\end{lemma}
\begin{proof}
The Fourier transform of the radial function $f$ is given by
$$
\wh{f}(\xi)=|\xi|^{-\frac{N-2}2}\int_0^\infty J_{\frac{N-2}2}(s|\xi|)f(s)s^{\frac{N}2}\, ds.
$$
Let $j^{(\ell)}$, $\ell\in N$ denote the positive zeros of the Bessel function $J_{\frac{N-2}2}$ of the first kind of order $\frac{N-2}2$, arranged in increasing order, 
and set $j^{(0)}:=0$. Then, it follows that $J_{\frac{N-2}2}>0$ in the interval $\bigl(j^{(2m-2)},j^{(2m-1)}\bigr)$ and
$J_{\frac{N-2}2}<0$ in the interval $\bigl(j^{(2m-1)},j^{(2m)}\bigr)$, $m\in\N$.
For $\xi\neq 0$, we can write therefore
\begin{align*}
&\int_0^\infty J_{\frac{N-2}2}(s|\xi|)f(s)s^{\frac{N}2}\, ds
= \sum_{\ell=1}^\infty \int_{\frac{j^{(\ell-1)}}{|\xi|}}^{\frac{j^{(\ell)}}{|\xi|}} s^\frac12J_{\frac{N-2}2}(s|\xi|) s^{\frac{N-1}2}f(s)\, ds\\
&\quad\geq \sum_{m=1}^{\infty}\left(\frac{j^{(2m-1)}}{|\xi|}\right)^{\frac{N-1}2} f\bigl(\frac{j^{(2m-1)}}{|\xi|}\bigr) 
\Bigl[ \int_{\frac{j^{(2m-2)}}{|\xi|}}^{\frac{j^{(2m-1)}}{|\xi|}} s^\frac12\bigl|J_{\frac{N-2}2}(s|\xi|)\bigr| ds 
-\int_{\frac{j^{(2m-1)}}{|\xi|}}^{\frac{j^{(2m)}}{|\xi|}} s^\frac12\bigl|J_{\frac{N-2}2}(s|\xi|)\bigr| ds \Bigr]\\
&\quad=\sum_{m=1}^{\infty}|\xi|^{-\frac32}\left(\frac{j^{(2m-1)}}{|\xi|}\right)^{\frac{N-1}2} f\bigl(\frac{j^{(2m-1)}}{|\xi|}\bigr) 
\Bigl[ \int_{j^{(2m-2)}}^{j^{(2m-1)}} t^\frac12\bigl|J_{\frac{N-2}2}(t)\bigr| dt 
-\int_{j^{(2m-1)}}^{j^{(2m)}} t^\frac12\bigl|J_{\frac{N-2}2}(t)\bigr| dt \Bigr],
\end{align*}
using the fact that $s\mapsto s^{\frac{N-1}2}f(s)$ is nonincreasing by assumption.
To conclude, an argument which goes back to Sturm \cite{sturm} (see also \cite{lorch-szego63,M}) shows that
\begin{equation}
  \label{eq:sturm-argument}
\int_{j^{(2m-2)}}^{j^{(2m-1)}} t^\frac12\bigl|J_{\frac{N-2}2}(t)\bigr| dt\geq  
\int_{j^{(2m-1)}}^{j^{(2m)}} t^\frac12\bigl|J_{\frac{N-2}2}(t)\bigr| dt,\quad\text{ for all }m\in\N,
\end{equation}
provided $N\geq 3$, and this gives the desired result. 
For the reader's convenience, we now give the proof of (\ref{eq:sturm-argument}).

\qquad Consider for $\nu>\frac12$ the function $z(t):=t^\frac12J_\nu(t)$. It satisfies $z(j^{(\ell)})=0$ and  $(-1)^\ell z'(j^{(\ell)})>0$ 
for all $\ell\in\N_0$. Moreover, it solves the differential equation
\begin{equation}\label{eqn:equa_diff_bessel}
z''(t) + \Bigl(1-\frac{\nu^2-\frac14}{t^2}\Bigr)z(t)=0\quad \text{for all $t>0.$}
\end{equation}
For $m\in\N$ and $t$ in the interval $I:=\bigl(j^{(2m-1)},\min\{j^{(2m)}, 2j^{(2m-1)}-j^{(2m-2)}\}\bigr)$, 
consider the functions $y_1(t)=-z(t)$ and $y_2(t)=z(2j^{(2m-1)}-t)$. 
According to the above remark, we have $y_1, y_2>0$ in $I$
and $y_1(j^{(2m-1)})=y_2(j^{(2m-1)})=0$. Moreover, $y_1'(j^{(2m-1)})=y_2'(j^{(2m-1)})\in(0,\infty)$.
Using the differential equation \eqref{eqn:equa_diff_bessel}, we find that
\begin{align*}
\frac{d}{dt}\left(y_1'(t)y_2(t)-y_1(t)y_2'(t)\right)&=y_1''(t)y_2(t)-y_1(t)y_2''(t)\\
&=(\nu^2-\frac14)\left(\frac1{t^2}-\frac1{(2j^{(2m-1)}-t)^2}\right)y_1(t)y_2(t)\\
&<0 \quad\text{for all }t\in I.
\end{align*}
Hence, 
\begin{equation}\label{eqn:phi_prime}
y_1'(t)y_2(t)-y_1(t)y_2'(t)<0\quad\text{ for all }j^{(2m-1)}<t\leq \min\{j^{(2m)}, 2j^{(2m-1)}-j^{(2m-2)}\},
\end{equation}
and since $y_2(2j^{(2m-1)}-j^{(2m-2)})=0$ and $y_2'(2j^{(2m-1)}-j^{(2m-2)})=-z'(j^{(2m-2)})<0$, the positivity
of $y_1$ in $I$ implies that $j^{(2m)}<2j^{(2m-1)}-j^{(2m-2)}$, i.e. $I=\bigl(j^{(2m-1)},j^{(2m)}\bigr)$.

\qquad Moreover, from \eqref{eqn:phi_prime}, we infer that the quotient $\frac{y_1}{y_2}$ is a decreasing
function in $I$ which vanishes at the right boundary of this interval. Consequently, $y_1(t)<y_2(t)$ in $I$,
i.e., $|z(t)|< |z(2j^{(2m-1)}-t)|$ for all $t\in(j^{(2m-1)},j^{(2m)})$ and we conclude that
$$
\int_{j^{(2m-2)}}^{j^{(2m-1)}}|z(t)|\, dt > \int_{j^{(2m-1)}}^{j^{(2m)}}|z(t)|\, dt.
$$
In the case $\nu=\frac12$, we have $z(t)=\sqrt{\frac2\pi}\sin t$ and $j^{(\ell)}=\ell\pi$, $\ell\in\N_0$.
Thus,
$$
\int_{j^{(\ell-1)}}^{j^{(\ell)}}|z(t)|\, dt =\sqrt{\frac2\pi}\int_0^\pi \sin t\, dt=2\sqrt{\frac2\pi}\quad\text{for all }\ell\in\N,
$$
and this concludes the proof of (\ref{eq:sturm-argument}).
\end{proof}

\qquad In our proof of the a priori bound given in Proposition~\ref{prop:apriori_defocusing}, we only need the following corollary of Lemma~\ref{lem:FT_rad_positive}.

\begin{corollary}\label{prop:bilinear_positive}
Let $N\geq 3$, $k>0$ and choose $\delta>0$ such that $k\delta\leq \zN$, where $\zN$ denotes
the first positive zero of the Bessel function $Y_{\frac{N-2}{2}}$. Then,
$$
\int_{\R^N}f(x) [(1_{B_\delta} \Psi_k)\ast f](x)\, dx\geq 0 \quad \text{for all }f\in L^{p'}(\R^N,\R), \ 2\leq p\leq 2^\ast,
$$ 
where $\Psi_k$ denotes the real part of the fundamental solution $\Phi_k$ defined in (\ref{eq:18}). 
\end{corollary}

\begin{proof}
Since $1_{B_\delta}\Psi_k\in L^1(\R^N)\cap L^{\frac{N}{N-2}}_{w}(\R^N)$, by the weak Young inequality
there is for each $2\leq p\leq 2^\ast$ a constant $C_p>0$ such that
$$
\left|\int_{\R^N}f(x) [(1_{B_\delta}\Psi_k)\ast f](x)\, dx\right|\leq C_p \|f\|_{p'}^2\quad \text{for all }f\in L^{p'}(\R^N,\R).
$$
Hence, it suffices to prove the conclusion for $f\in\cS(\R^N,\R)$. For such functions, Parseval's identity gives
\begin{equation}\label{eqn:parseval}
\int_{\R^N}f(x) [(1_{B_\delta}\Psi_k)\ast f](x)\, dx
=(2\pi)^{\frac{N}2}\int_{\R^N} |\wh{f}(\xi)|^2 \cF\bigl(1_{B_\delta}\Psi_k\bigr)(\xi)\, d\xi.
\end{equation}

\qquad It thus remains to show that 
\begin{equation}
  \label{eq:positivity-Fourier-proof}
\cF\bigl(1_{B_\delta}\Psi_k\bigr) \ge 0 \quad \text{on  $\R^N$.}   
\end{equation}
In the radial variable, the radial function 
$1_{B_\delta}\Psi_k$ is given, up to a positive constant factor, by $t \mapsto -t^{\frac{2-N}{2}}1_{[0,\delta]}(t)Y_{\frac{N-2}2}(kt)$. Moreover, for $N\geq 3$ the function
$t\mapsto t^\frac12Y_{\frac{N-2}2}(kt)$ is negative and increasing on $(0,\delta)$. Hence Lemma~\ref{lem:FT_rad_positive} implies (\ref{eq:positivity-Fourier-proof}), and the proof is finished.
\end{proof}

\qquad We can now  prove Proposition~\ref{prop:apriori_defocusing}.

\begin{proof}[Proof of Proposition~\ref{prop:apriori_defocusing}]
We write $u:=v+\varphi$ and $u=u_1+iu_2$ with real-valued functions $u_1, u_2\in L^p_\text{loc}(\R^N)$.
Multiplying the equation \eqref{eqn:fp_complex} by $Q|u|^{p-2}\overline{u}$ and
integrating over $\R^N$, we find
\begin{align*}
&\int_{\R^N}Q|u|^p\, dx-\int_{\R^N}Q|u|^{p-2}\varphi\overline{u}\, dx \\
&\quad= \int_{\R^N}Q|u|^{p-2}(u_1-iu_2)[\Phi_k\ast\bigl(Q|u|^{p-2}(u_1+iu_2)\bigr)]\, dx \\
&\quad= \int_{\R^N}Q|u|^{p-2}u_1[\Phi_k\ast\bigl(Q|u|^{p-2}u_1\bigr)]\, dx 
+\int_{\R^N}Q|u|^{p-2}u_2[\Phi_k\ast\bigl(Q|u|^{p-2}u_2\bigr)]\, dx \\
&\qquad + i \int_{\R^N}Q|u|^{p-2}u_1[\Phi_k\ast\bigl(Q|u|^{p-2}u_2\bigr)]\, dx 
- i\int_{\R^N}Q|u|^{p-2}u_2[\Phi_k\ast\bigl(Q|u|^{p-2}u_1\bigr)]\, dx \\
&\quad=\int_{\R^N}Q|u|^{p-2}u_1[\Phi_k\ast\bigl(Q|u|^{p-2}u_1\bigr)]\, dx
 +\int_{\R^N}Q|u|^{p-2}u_2[\Phi_k\ast\bigl(Q|u|^{p-2}u_2\bigr)]\, dx,
\end{align*}
where the symmetry of the convolution has been used in the last step.
Taking real parts on both sides of the equality, we obtain
\begin{equation}\label{eqn:integr_estim1}
\begin{aligned}
\int_{\R^N}Q|u|^p\, dx-\int_{\R^N}Q|u|^{p-2}\text{Re}\left(\varphi\overline{u}\right)\, dx 
&=\int_{\R^N}Q|u|^{p-2}u_1[\Psi_k\ast\bigl(Q|u|^{p-2}u_1\bigr)]\, dx\\
&\quad+\int_{\R^N}Q|u|^{p-2}u_2[\Psi_k\ast\bigl(Q|u|^{p-2}u_2\bigr)]\, dx.
\end{aligned}
\end{equation}
where again $\Psi_k$ denotes the real part of $\Phi_k$. 
Notice in addition that setting $\delta=\text{diam}(\text{supp }Q)$, the assumption (Q2) 
implies $\delta\le \frac{\zN}{k}$ and hence, for all $f\in L^{p'}_{\text{loc}}(\R^N)$,
$$
\int_{\R^N}Qf[\Psi_k\ast (Qf)]\, dx = \int_{\R^N}Qf[(1_{B_\delta}\Psi_k)\ast(Qf)]\, dx\geq 0,
$$
by Corollary~\ref{prop:bilinear_positive}. Thus, as a consequence of \eqref{eqn:integr_estim1}, we find
$$
\int_{\R^N}Q|u|^p\, dx\geq \int_{\R^N}Q|u|^{p-2}\text{Re}\left(\overline{u}\varphi\right)\, dx,
$$
and, since $Q\leq 0$ on $\R^N$, by (Q1), it follows that
\begin{equation}\label{eqn:first_bound}
\int_{\R^N}|Q|\ |u|^p\, dx \leq \|\varphi\|_\infty \int_{\R^N} |Q|\  |u|^{p-1}\, dx.
\end{equation}
Using H\"older's inequality we then obtain the estimate
\begin{align*}
\int_{\R^N} |Q|\  |u|^{p-1}\, dx & \leq \left(\int_{\R^N} |Q|\, dx\right)^{\frac1p}\left(\int_{\R^N}|Q|\ |u|^{p}\, dx\right)^{\frac1{p'}}\\
& \leq \left(\int_{\R^N} |Q|\, dx\right)^{\frac1p} \left(\|\varphi\|_\infty \int_{\R^N} |Q|\  |u|^{p-1}\, dx\right)^{\frac1{p'}},
\end{align*}
and therefore
$$
\int_{\R^N} |Q|\  |u|^{p-1}\, dx \leq \|\varphi\|_\infty^{p-1}
\int_{\R^N}|Q|\, dx \leq |\Omega|\ \|Q\|_\infty\ \|\varphi\|_\infty^{p-1},
$$
where $\Omega=\{x\in\R^N\, :\, Q(x)\neq 0\}$.
Using again \eqref{eqn:first_bound}, we deduce that
$$
\|\ |Q|^{\frac1{p'}}\ |u|^{p-1}\|_{p'}^{p'}
=\int_{\R^N}|Q|\ |u|^p\, dx
\leq |\Omega|\ \|Q\|_\infty\ \|\varphi\|_\infty^p.
$$
Since the support $Q$ is compact and since $p<2^\ast$, 
H\"olders inequality yields the estimates
\begin{align}
\|Q|u|^{p-1}\|_{(2^\ast)'} \leq |\Omega|^{\frac1{(2^\ast)'}-\frac1{p'}} \|Q|u|^{p-1}\|_{p'}
&\leq |\Omega|^{\frac1{(2^\ast)'}-\frac1{p'}} \|Q\|_\infty^{\frac1p} \|\ |Q|^{\frac1p'}|u|^{p-1}\|_{p'} \nonumber\\
&\leq |\Omega|^{\frac1{(2^\ast)'}} \|Q\|_\infty \|\varphi\|_\infty^{p'-1}=:D.
\label{eqn:final_bound}
\end{align}
Lemma~\ref{lem:regularity1} with $a=Q$ and the 
estimate~\eqref{eqn:final_bound} imply the
existence of constants
$C=C(N,k,p,\|Q\|_\infty,|\Omega|)>0$ and $m=m(N,p)\in\N$ such that
\begin{align*}
\|v\|_\infty\leq C\left(D+D^{(p-1)^m}+\|\varphi\|_\infty^{p-1} + \|\varphi\|_\infty^{(p-1)^m}\right).
\end{align*}
Making $C>0$ larger if necessary, we thus obtain~(\ref{eq:apriori-superlinear-defocusing-estimate}), as claimed.
\end{proof}

\section{Proofs of the main results}
\label{sec:proofs-main-results}

In this section, we complete the proofs of the main results in the introduction.

\begin{proof}[Proof of Theorem~\ref{W teo 1-sublinear}]
Let $\phi \in X:= L^\infty(\R^N)$. We write (\ref{nlh-1-integral}) as a fixed point equation
$$
u = \cA(u) \qquad \text{in $X$}  
$$
with the nonlinear operator 
\begin{equation}
  \label{eq:A-operator}
\cA: X \to X, \qquad \cA[w]=\cR_k (N_f(w))+ \phi.
\end{equation}
Since $\alpha > \frac{N+1}{2}$, we may fix $\alpha' \in (\frac{N+1}{2}, \alpha)$. By Lemma~\ref{lem:nemytskii_cont}, the nonlinear operator $N_f: X \to L^\infty_{\alpha'}(\R^N)$ is well-defined and continuous. Moreover, $\cR_k: L^\infty_{\alpha'}(\R^N) \to X$ is compact by Proposition~\ref{resolvent-compact-and-continuous}. Consequently, $\cA$ is a compact and continuous operator. Moreover, the set 
$$
\cF:=\{u\in X: \, u=\mu  \cA[u]\ \text{ for some } \mu\in[0,1]\}
$$
is bounded by Proposition~\ref{sec:proof-theorem-refw-1}. Hence Schaefer's fixed point theorem (see e.g. \cite[Chapter 9.2.2.]{Evans}) implies that $\cA$ has a fixed point. 
\end{proof}

\qquad We continue with the proof of Theorem~\ref{thm:rabinowitz-applied}. For this we recall the following variant of Rabinowitz' 
global continuation theorem (see \cite[Theorem 3.2]{rabinowitz71}; 
see also \cite[Theorem 14.D]{zeidler}).
\begin{theorem}\label{thm:rabinowitz}
	Let $(X,\|\cdot\|)$ be a real Banach space, and consider a continuous and compact mapping $G$: $\R\times X$ $\to$ $X$ satisfying $G(0,0)=0$.

Assume that 
	\begin{itemize}
		\item[(a)] $G(0,u)=u$ $\Leftrightarrow$ $u=0$, and
		\item[(b)] there exists $r>0$ such that $\text{\em deg}(id-G(0,\cdot),B_r(0),0)\neq 0$,
		where $\text{\em deg}$ denotes the Leray-Schauder degree.
	\end{itemize}
 
Moreover, denote by $S$ the set of solutions $(\lambda, u)\in \R\times X$ of the equation 
$$
u=G(\lambda,u).
$$
Then the connected components $C^+$ and $C^-$ of $S$ in $[0,\infty)\times X$ and $(-\infty,0]\times X$ 
	which contain $(0,0)$ are both unbounded.
\end{theorem}

\begin{proof}[Proof of Theorem~\ref{thm:rabinowitz-applied} (completed)]
Let $2<p<2^\ast$, $Q\in L^\infty_\alpha(\R^N,\R)\backslash\{0\}$ for some $\alpha > \frac{N+1}{2}$, 
$\phi\in X:= L^\infty(\R^N)$
and consider $G$: $\R \times X \to X$ given by
\begin{equation}
  \label{eq:def-G-function}
G(\lambda,w)= \cR_k\bigl(Q|w|^{p-2}w\bigr)+\lambda\phi,
\end{equation}
Using Proposition~\ref{resolvent-compact-and-continuous} and Lemma~\ref{lem:nemytskii_cont}, we obtain that the map $G$ is continuous and compact.

Moreover, if $w\in X$ satisfies $w=G(\lambda,w)$, then $w$ is a solution of (\ref{nlh-1-integral-parameter-lambda}). 

Furthermore, if $w \in X$ satisfies $w = G(0,w) = \cR_k\bigl(Q|w|^{p-2}w\bigr)$, then $w = 0$ by Proposition~\ref{corol-kato}. 

\qquad To compute the Leray-Schauder degree, we remark that $G(0,0)=0$ and $\partial_wG(0,0)=0$ by Lemma~\ref{lem:nemytskii_C-1}.  
Hence, we can find some radius $r>0$ 
such that $\|G(0,w)\|_{L^\infty} \leq \frac12 \|w\|_{L^\infty}$
for all $w\in X$ such that $\|w\|_{L^\infty} \leq r$. 
Therefore, the compact homotopy $H(t,w)=tG(0,w)$ is admissible in the ball
$B_r(0)\subset X$
and we find that 
\begin{align*}
\text{deg}(id-G(0,\cdot),B_r(0),0)=\text{deg}(id-H(1,\cdot),B_r(0),0)
&=\text{deg}(id-H(0,\cdot),B_r(0),0)\\
&=\text{deg}(id,B_r(0),0)=1.
\end{align*}
Theorem~\ref{thm:rabinowitz} therefore applies and 
we obtain the existence of an unbounded branch 
$C_\phi \subseteq \bigl\{(\lambda,w)\in \R\times X\, :\,
w=G(\lambda,w)\text{ and } \lambda\geq 0\bigr\}$
which contains $(0,0)$. Moreover, $C_\phi \setminus \{(0,0)\}$ is a subset of $(0,\infty) \times X$ since $w=G(0,w)$ implies $w=0$ by Proposition~\ref{corol-kato}, as noted above. 
\end{proof}

\begin{remark}
\label{remark-rabinowitz}
The application of Theorem~\ref{thm:rabinowitz} to the function $G$ defined in (\ref{eq:def-G-function}) also yields a connected component 
$$
C_\phi^- \subset \bigl\{(\lambda,w)\in \R\times X\, :\,
w=G(\lambda,w)\text{ and } \lambda\le 0\bigr\}
$$
which contains $(0,0)$. However, this component is also obtained by passing from $\phi$ to $-\phi$ in the statement of Theorem~\ref{thm:rabinowitz-applied}, since by definition we have $C_\phi^- = C_{-\phi}$.
\end{remark}

\qquad  We may now also prove  Theorem~\ref{thm:unbounded_branch-defocusing}. 

\begin{proof}[Proof of Theorem~\ref{thm:unbounded_branch-defocusing}]
Since, by assumption, $Q\leq 0$ in $\R^N$ 
and $\text{diam}(\text{supp }Q)\le \frac{\zN}{k}$, 
the a priori bounds in Proposition~\ref{prop:apriori_defocusing} 
imply that the unbounded branch $C_\phi$ contains, for each $\lambda \ge 0$,
at least one pair $(\lambda,w)$, as claimed.
\end{proof}

\qquad  Next, we complete Theorem~\ref{teo-implicit-function}.

\begin{proof}[Proof of Theorem~\ref{teo-implicit-function}]
Let again $X:= L^\infty(\R^N)$, and consider the nonlinear operator $\cB: X \to X$, $\cB(u):= u- \cR_k N_f(u)$. Then $\cB(0)=0$, since $N_f(0)=0$ by assumption. Since $N_f: X \to L^\infty_{\alpha'}$ is differentiable by Lemma~\ref{lem:nemytskii_C-1}, $\cB$ is differentiable as well. Moreover 
$$
\cB'(0)= \id - \cR_k N_f'(0) = \id \in \cL_{\R}(X,X),
$$
since $N_f'(0) =N_{f'}(0) =0 \in \cL_{\R}(X,L^\infty_{\alpha'})$ by assumption and Lemma~\ref{lem:nemytskii_C-1}. Consequently, 
$\cB$ is a diffeomorphism between open neighborhoods $U,V \subset X$ of zero, and this shows the claim. 
\end{proof}

\qquad Finally, we state and prove the unique existence of solutions in the case where $f$ satisfies a suitable Lipschitz condition.  

\begin{theorem} \label{theo-uniqueness}
Let, for some $\alpha>\frac{N+1}{2}$, the nonlinearity $f: \R^N \times \C \to \C$ be a continuous function satisfying (\ref{eq:assumption-f1}) and the Lipschitz condition 
\begin{equation}
  \label{eq:assumption-f1-lipschitz2}
\ell_\alpha:= \sup \Bigl \{ \langle x \rangle^{\alpha} \, \Bigl|\frac{f(x,u)-f(x,v)}{u-v}\Bigr|\::\: u,v \in \R, \: x \in \R^N \Bigr\} < \frac{1}{\kappa_\alpha},   
\end{equation}
where $\kappa_\alpha$ is defined in Proposition~\ref{resolvent-compact-and-continuous}. 

Then, for any given solution $\phi \in L^\infty(\R^N)$ of the homogeneous Helmholtz equation $\Delta \phi + k \phi = 0$, the equation (\ref{nlh-1-integral}) admits precisely one solution $u \in L^\infty(\R^N)$. 
\end{theorem}

\begin{proof}
Let $\phi \in X:= L^\infty(\R^N)$. As in the proof of Theorem~\ref{W teo 1-sublinear} given above, we write (\ref{nlh-1-integral}) as a fixed point equation
$u = \cA(u)$ in $X$ with the nonlinear operator $\cA$ defined in (\ref{eq:A-operator}). Assumption (\ref{eq:assumption-f1-lipschitz2}) implies that
$$
\|\cA(u)-\cA(v)\|_{X} = \bigl \|\cR_k \bigl (N_f(u)-N_f(v)\bigr)\bigr\| \le \kappa_{\alpha} \|N_f(u)-N_f(v)\|_{L^\infty_\alpha}\le \kappa_\alpha \ell_\alpha \|u-v\|_X
$$
with $\kappa_\alpha \ell_\alpha<1$. Hence $\cA$ is a contraction, and thus it has a unique fixed point in $X$. 
\end{proof}

\appendix

\section{Uniform regularity estimates}
\label{sec:unif-regul-estim}

In this section, we wish to prove uniform regularity estimates for solutions of (\ref{nlh-1-integral}) in the case where the nonlinearity $f$ is of the form given in (\ref{eq:power-type}). These estimates, which we used in the proof of the a priori bound given in Proposition~\ref{prop:apriori_defocusing},  allow to pass from uniform bounds in $L^{(2^*)'}(\R^N)$ to uniform bounds in $L^\infty(\R^N)$. 
The proof of the following lemma is similar to a regularity estimate for real-valued solutions given in \cite[Proposition 3.1]{EW2}, but the differences justify to include a complete proof in this paper. 

In the following, for $q \in [1,\infty]$, we let $L^q_c(\R^N)$ denote the space of functions in $L^q(\R^N)$ with compact support in $\R^N$.

\begin{lemma}\label{lem:regularity1}
Let $N\geq 3$, $2<p<2^\ast$ and consider a function $a\in L^\infty_c(\R^N)$.

For $k>0$ and $\varphi\in L^\infty_{\text{loc}}(\R^N)$, every solution $v\in L^p_{\text{loc}}(\R^N)$ of
$$
v=\Phi_k \ast \bigl(a|v|^{p-2}v \bigr)+ \phi
$$
satisfies $v\in W^{2,t}(\R^N)$ for all $2_\ast\leq t<\infty$.
In particular, $u\in L^\infty(\R^N)$ and there exist constants
$$
C=C\bigl(N,k,p,\|a\|_\infty\bigr)>0\qquad \text{and}\qquad m=m(N,p)\in\N
$$
independent of $v$ and $\varphi$ such that
\begin{equation}\label{eqn:infty-estimate}
\|v\|_\infty\leq C\left( \|a |\varphi|^{p-1}\|_{(2^\ast)'}+\|a|v|^{p-1}\|_{(2^\ast)'}^{(p-1)^m}
+\|\varphi\|_{\infty}^{p-1}+\|\varphi\|_{\infty}^{(p-1)^m}\right).
\end{equation}

\end{lemma}


\noindent{\bf Proof.} Since, by assumption, $v \in L^p_{\text{loc}}(\R^N)$, and since $a\in L^\infty_c(\R^N)$, it follows that
\begin{equation}\label{eqn:f1_f2_Lpq}
f:=a|v |^{p-2}v\in L^q_c(\R^N), \quad\text{ for all }1\leq q\leq p'.
\end{equation}
Furthermore, since $v=\Phi_k\ast f + \phi$, we deduce that
\begin{equation}\label{eqn:f1_f2_estim_Rf}
|f|\leq 2^{p-2} |a| \bigl(|\Phi_k\ast f|^{p-1} + |\varphi|^{p-1}\bigr) \quad\text{a.e. in }\R^N.
\end{equation}

\qquad We start by proving that $v\in L^\infty(\R^N)$. For this, we first remark
that $f\in L^{(2^\ast)'}_{c}(\R^N)$, since $p<2^\ast$.
Consequently, the mapping properties of $\Phi_k$ given in \cite[Proposition A.1]{EW1}
yield $\Phi_k\ast f\in L^{2^\ast}(\R^N)\cap W^{2,(2^\ast)'}_{\text{loc}}(\R^N)$ and, for every $0<R<2$, the existence of constants $\tilde{C}_0=\tilde{C}_0(N,k,R)>0$ and $D=D(N,k)>0$ such that
\begin{align*}
\|\Phi_k\ast f\|_{W^{2,(2^\ast)'}(B_{R}(x_0))}
&\leq \tilde{C}_0 \left( \|\Phi_k\ast f\|_{L^{(2^\ast)'}(B_2(x_0))}+\|f\|_{L^{(2^\ast)'}(B_2(x_0))}\right)\\
&\leq \tilde{C}_0(D+1)\|f\|_{(2^\ast)'} \quad\text{for all }x_0\in\R^N.
\end{align*}

\qquad  Setting $C_0:=\tilde{C}_0(D+1)$, we consider a strictly decreasing sequence
$2>R_1>R_2>\ldots>R_j>R_{j+1}>\ldots>1$.
From Sobolev's embedding theorem, there is for each $1\leq t\leq 2^\ast$,
a constant $\kappa_t^{(0)}=\kappa_t^{(0)}(N,t)>0$ such that
$$
\|\Phi_k\ast f\|_{L^t(B_{R_1}(x_0))}\leq \kappa_t^{(0)} C_0 \|f\|_{(2^\ast)'},
$$
where $C_0$ is given as above, with $R=R_1$.
Choosing $t_1:=\frac{2^\ast}{p-1}$, we obtain from \eqref{eqn:f1_f2_estim_Rf}, there is some constant $D_2=D_2(N,p)>0$ such that
\begin{align*}
\|f\|_{L^{t_1}(B_{R_1}(x_0))}
&\leq D_2 \|a\|_\infty \bigl(\|\Phi_k\ast f\|_{L^{2^\ast}(B_{R_1}(x_0))}^{p-1}+\|\varphi\|_{L^{2^\ast}(B_{R_1}(x_0))}^{p-1}\bigr)\\
&\leq D_2 \|a\|_\infty\left((\kappa_{2^\ast}^{(0)} C_0)^{p-1} \|f\|_{(2^\ast)'}^{p-1}+|B_{R_1}|^{\frac1{t_1}}\|\varphi\|_\infty^{p-1}\right).
\end{align*}
\qquad It then follows as in \cite[Proof of Proposition A.1(i)]{EW1} from elliptic regularity theory that $\Phi\ast f\in W^{2,t_1}_{\text{loc}}(\R^N)$ and for some
constant $\tilde{C}_1=\tilde{C}_1(N,k,p)>0$,
\begin{align*}
\|\Phi_k\ast f&\|_{W^{2,t_1}(B_{R_2}(x_0))}
 \leq \tilde{C}_1 \left( \|\Phi_k\ast f\|_{L^{t_1}(B_{R_1}(x_0))}+\|f\|_{L^{t_1}(B_{R_1}(x_0))}\right)\\
&\ \ \leq \tilde{C}_1\Bigl[\kappa_{t_1}^{(0)}C_0\|f\|_{(2^\ast)'}
+D_2 \|a\|_\infty\left((\kappa_{2^\ast}^{(0)} C_0)^{p-1} \|f\|_{(2^\ast)'}^{p-1}
+|B_{R_1}|^{\frac1{t_1}}\|\varphi\|_\infty^{p-1}\right)\Bigr]\\
&\ \ \leq C_1\left( \|f\|_{(2^\ast)'}+\|f\|_{(2^\ast)'}^{p-1}+\|\varphi\|_\infty^{p-1}\right)\qquad\text{for all }x_0\in\R^N,
\end{align*}
where $C_1=C_1\bigl(N,k,p,\|a\|_\infty\bigr)$.
  If $t_1\geq \frac{N}{2}$, Sobolev's embedding theorem gives for each $1\leq t<\infty$ the existence of a constant
$\kappa^{(1)}_t=\kappa_t^{(1)}(N,q,t)>0$ such that
$$
\|\Phi_k\ast f\|_{L^t(B_{R_2}(x_0))}\leq \kappa^{(1)}_t C_1\left( \|f\|_{(2^\ast)'}+\|f\|_{(2^\ast)'}^{p-1}+\|\varphi\|_\infty^{p-1}\right).
$$
As a consequence, we obtain
\begin{align*}
\|f\|_{L^t(B_{R_2}(x_0))}
&\leq D_2 \|a\|_\infty\Bigl(3^{p-2}(\kappa_{t(p-1)}^{(1)} C_1)^{p-1} \left(\|f\|_{(2^\ast)'}^{p-1}+\|f\|_{(2^\ast)'}^{(p-1)^2}
+\|\varphi\|_\infty^{(p-1)^2}\right)\\&\quad +|B_{R_2}|^{\frac{p-1}{t}}\|\varphi\|_\infty^{p-1}\Bigr),
\end{align*}
for all $1\leq t<\infty$. As in \cite[Proof of Proposition A.1(i)]{EW1}, it then follows from elliptic regularity theory that 
$\Phi\ast f\in W^{2,N}_{\text{loc}}(\R^N)$, and since $R_2>1$, there exists some constant $\tilde{C}_2=\tilde{C}_2(N,k)>0$ such that
\begin{align*}
\|\Phi_k\ast f&\|_{W^{2,N}(B_1(x_0))}\leq \tilde{C}_2 \left( \|\Phi_k\ast f\|_{L^N(B_{R_2}(x_0))}+\|f\|_{L^N(B_{R_2}(x_0))}\right)\\
&\leq \tilde{C}_2\Bigl\{\kappa_N^{(1)}C_1\left(\|f\|_{(2^\ast)'}+\|f\|_{(2^\ast)'}^{p-1}+\|\varphi\|_\infty^{p-1}\right)\\
&+D_2 \|a\|_\infty\Bigl(3^{p-2}(\kappa_{N(p-1)}^{(1)} C_1)^{p-1} \left(\|f\|_{(2^\ast)'}^{p-1}+\|f\|_{(2^\ast)'}^{(p-1)^2}
+\|\varphi\|_\infty^{(p-1)^2}\right)\\
&+|B_{R_2}|^{\frac{p-1}N}\|\varphi\|_\infty^{p-1}\Bigr)\Bigr\}\\
&\leq C_2\left( \|f\|_{(2^\ast)'}+\|f\|_{(2^\ast)'}^{(p-1)^2}+\|\varphi\|_\infty^{p-1}+\|\varphi\|_\infty^{(p-1)^2}\right)
\end{align*}
for all $x_0\in\R^N$, where $C_2=C_2\bigl(N,k,p,\|a\|_\infty\bigr)$.
By Sobolev's embedding theorem, there is a constant $\kappa_\infty=\kappa_\infty(N)>0$ such that
$$
\|\Phi_k\ast f\|_{L^\infty(B_1(x_0))}
\leq  \kappa_\infty C_2\left( \|f\|_{(2^\ast)'}+\|f\|_{(2^\ast)'}^{(p-1)^2}+\|\varphi\|_\infty^{(p-1)^2}+\|\varphi\|_\infty^{p-1}\right)
$$
for all $x_0\in\R^N$. Therefore, $\Phi\ast f\in L^\infty(\R^N)$ and since $v=\Phi\ast f$, the estimate \eqref{eqn:infty-estimate} holds
with $C=2\kappa_\infty C_2$ and $m=2$.

\qquad If $t_1<\frac{N}{2}$, we infer from Sobolev's embedding theorem that
$$
\|\Phi_k\ast f\|_{L^t(B_{R_2}(x_0))}\leq \kappa^{(1)}_t C_1\left(\|f\|_{(2^\ast)'}
+\|f\|_{(2^\ast)'}^{p-1}+\|\varphi\|_\infty^{p-1}\right)
$$
for each $1\leq t\leq \frac{Nt_1}{N-2t_1}$, where $\kappa_t^{(1)}=\kappa_t^{(1)}(N,p,t)$.
Therefore, setting $t_2:=\frac{Nt_1}{(N-2t_1)(p-1)}$, we obtain from \eqref{eqn:f1_f2_estim_Rf},
\begin{align*}
&\|f\|_{L^{t_2}(B_{R_2}(x_0))}
\\&\leq D_2 \|a\|_\infty\Bigl(3^{p-2}(\kappa_{t_2(p-1)}^{(1)} C_1)^{p-1} \left(\|f\|_{(2^\ast)'}^{p-1}
+\|f\|_{(2^\ast)'}^{(p-1)^2}+\|\varphi\|_\infty^{(p-1)^2}\right)+|B_{R_2}|^{\frac{p-1}{t_2}}\|\varphi\|_\infty^{p-1}\Bigr).
\end{align*}
Using again elliptic regularity theory as before, we find that $\Phi_k\ast f\in W^{2,t_2}_{\text{loc}}(\R^N)$
and for some constant $\tilde{C}_2=\tilde{C}_2(N,k,p)>0$,
\begin{align*}
\|\Phi_k\ast f&\|_{W^{2,t_2}(B_{R_3}(x_0))}\leq \tilde{C}_2 \left( \|\Phi_k\ast f\|_{L^{t_2}(B_{R_2}(x_0))}+\|f\|_{L^{t_2}(B_{R_2}(x_0))}\right)\\
\ \ &\leq \tilde{C}_2\Bigl\{\kappa_{t_2}^{(1)}C_1\left(\|f\|_{(2^\ast)'}+\|f\|_{(2^\ast)'}^{p-1}+\|\varphi\|_\infty^{p-1}\right)\\
& \ \  \quad +D_2 \|a\|_\infty\Bigl(3^{q-2}(\kappa_{t_2(p-1)}^{(1)} C_1)^{p-1} \left(\|f\|_{(2^\ast)'}^{p-1}+\|f\|_{(2^\ast)'}^{(p-1)^2}+\|\varphi\|_\infty^{(p-1)^2}\right)\\
&\ \  \quad +|B_{R_2}|^{\frac{p-1}{t_2}}\|\varphi\|_\infty^{p-1}\Bigr)\Bigr\}\\
&\ \    \leq C_2\left( \|f\|_{(2^\ast)'}+\|f\|_{(2^\ast)'}^{(p-1)^2}+\|\varphi\|_\infty^{p-1}+\|\varphi\|_\infty^{(p-1)^2}\right),
\end{align*}
for all $x_0\in\R^N$, where $C_2=C_2\bigl(N,k,p,\|a\|_\infty\bigr)$.

\qquad Remarking that $t_2>t_1$, since $p<2^\ast$, we may iterate the procedure.
At each step we find some constant $C_j=C_j\bigl(N,k,p,\|a\|_\infty\bigr)$ such that the estimate
$$
\|\Phi_k\ast f\|_{W^{2,t_j}(B_{R_{j+1}}(x_0))}
\leq C_j\left(\|f\|_{(2^\ast)'}+\|f\|_{(2^\ast)'}^{(p-1)^j}+\|\phi\|_\infty^{p-1}+\|\phi\|_\infty^{(p-1)^j}\right)
$$
holds and where $t_j$ is defined recursively via $t_0=(2^\ast)'$ and
$t_{j+1}=\frac{Nt_j}{(N-2t_j)(p-1)}$, as long as $t_j<\frac{N}{2}$.
Since $t_{j+1}\geq \frac{t_1}{p'}\,t_j$ and since $t_1>p'$, we reach after finitely many steps
$t_\ell\geq\frac{N}{2}$, where
$\ell$ only depends on $N$ and $p$.
Since $R_j>1$ for all $j$, using the regularity properties of $\Phi $
and arguing as above, we obtain $\Phi\ast f\in W^{2,N}_{\text{loc}}(\R^N)$ as well as the estimate
$$
\|\Phi_k\ast f\|_{W^{2,N}(B_1(x_0))}\leq C_{\ell+1}\left(\|f\|_{(2^\ast)'}+\|f\|_{(2^\ast)'}^{(p-1)^{\ell+1}}
+\|\varphi\|_\infty^{p-1}+\|\varphi\|_\infty^{(p-1)^{\ell+1}}\right),  $$
where $x_0$ is any point of $\R^N$ and  $C_{\ell+1}=C_{\ell+1}\bigl(N,k,p,\|a\|_\infty\bigr)$
is independent of $x_0$.
Then, Sobolev's embedding theorem gives a constant $\kappa_\infty=\kappa_\infty(N)$ for which
$$
\|\Phi_k\ast f\|_{L^\infty(B_1(x_0))}
\leq \kappa_\infty C_{\ell+1}\left(\|f\|_{(2^\ast)'}+\|f\|_{(2^\ast)'}^{(q-1)^{\ell+1}}
+\|\varphi\|_\infty^{p-1}+\|\varphi\|_\infty^{(p-1)^{\ell+1}}\right)
$$
holds for all $x_0\in\R^N$. Hence, $\Phi\ast f\in L^\infty(\R^N)$ and choosing
$C=\kappa_\infty C_{\ell+1}$ and $m=\ell+1$ concludes the proof of \eqref{eqn:infty-estimate}.
We complete the proof. \hfill$\Box$\medskip

\bigskip

\noindent{\bf Acknowledgements:}  H. Chen is   supported by NNSF of China, No: 12071189,  12001252,
by the Jiangxi Provincial Natural Science Foundation, No: 20202BAB201005, 20202ACBL201001  and by the Alexander von Humboldt Foundation.    T. Weth is supported by the German Science Foundation (DFG) within the project WE-2821/5-2.

\bibliographystyle{amsplain}

\end{document}